\documentclass[11pt,a4paper]{article}
\usepackage[utf8]{inputenc}

\usepackage{amssymb,amscd,amsfonts,amsbsy,amsmath,amsthm,amsrefs}
\usepackage{dsfont}
\usepackage{enumerate}
\usepackage{mathrsfs}
\usepackage{epsf,epsfig,esint}
\usepackage{pdfsync}
\usepackage[all]{xy}
\usepackage{pdfpages}
\usepackage{xcolor}
\usepackage{stmaryrd}
\usepackage[colorlinks=true,citecolor=blue,linkcolor=cyan]{hyperref}
\usepackage{euscript}
\newtheorem{proposition}{Proposition}[section]
\newtheorem{theorem}[proposition]{Theorem}
\newtheorem{lemma}[proposition]{Lemma}

\theoremstyle{definition}
\newtheorem{definition}[proposition]{Definition}
\newtheorem{notation}[proposition]{Notation}

\theoremstyle{remark}
\newtheorem{remark}[proposition]{Remark}

\numberwithin{equation}{section}

\newcommand{\Ri}{\mathds{R}}

\newcommand{\Leray}{\mathds{P}}

\newcommand{\e}{\varepsilon}
\newcommand{\less}{\lesssim}

\newcommand{\cb}{{\cal B}}		\newcommand{\CB}{{\mathscr{B}}}
		\newcommand{\CC}{{\mathscr{C}}}

		\newcommand{\CS}{{\mathscr{S}}} 

		\newcommand{\CU}{{\mathscr{U}}}

\newcommand{\curl}{\mathop{\rm curl}}

\newcommand{\divv}{\mathop{\rm div}}

\newcommand{\Sh}{S}
\newcommand{\Mh}{M}
\newcommand{\Dh}{D}
\newcommand{\Lh}{\Delta}

\newcommand{\dd}{\, {\rm d}}

\newcommand{\contr}{\lrcorner \,}
\newcommand{\Nn}{{\rm{\sf N}}}
\newcommand{\Rr}{{\rm{\sf R}}}

\newcommand{\demi}[1]{[#1[}
\newcommand{\ouvert}[1]{]#1[}

\color{black}
\author{Cl\'ement Denis \\
Aix-Marseille University, I2M}
\date{}

	\title{Existence and uniqueness in critical spaces for the magnetohydrodynamical system in $\Ri^n$}

\begin{document}
	
\maketitle
	
	\begin{abstract}
		We give a description of a magnetohydrodynamical system in $n$ dimension using the exterior derivative. We then prove existence of global solutions for small initial data and local existence for arbitrary large data in two classes of critical spaces - $L^q_tL^p_x$ and $\CC_tL^p_x$, as well as uniqueness for solutions in $\CC_tL^p_x$.
	\end{abstract}

\section{Introduction}

In $\Ri^3$, the magnetohydrodynamical system on a time interval $\ouvert{0,T}$ ($0<T\le+\infty$) as considered in \cite{ST83} and \cite{Mo21} is written as
\begin{equation}
	\label{mhd}
	\left\{
	\begin{array}{rclcl}
		\partial_t u-\Delta u+\nabla \pi-u\times({\rm curl}\,u)&=&({\rm curl}\,b)\times b&\mbox{ in }&\ouvert{0,T}\times\Ri^3\\
		\partial_t b-\Delta b&=&{\rm curl}\,(u\times b)&\mbox{ in }&\ouvert{0,T}\times\Ri^3\\
		{\rm div}\,u&=&0&\mbox{ in }&\ouvert{0,T}\times\Ri^3\\
		{\rm div}\,b&=&0&\mbox{ in }&\ouvert{0,T}\times\Ri^3\\
	\end{array}
	\right.
\end{equation}
where the {\sl velocity} of the (incompressible homogeneous) fluid is denoted by $u:\,\ouvert{0,T}\times\Ri^3\to{\mathds{R}}^3$, the {\sl magnetic field} is denoted by $b:\,\ouvert{0,T}\times\Ri^3\to{\mathds{R}}^3$ and the (dynamic) {\sl pressure} of the fluid is denoted by $\pi:\,\ouvert{0,T}\times\Ri^3\to{\mathds{R}}$. 

The first equation in \eqref{mhd} corresponds to the {\sl Navier-Stokes}
equations with the fluid subject to
the {\sl Laplace force} $({\rm curl}\,b)\times b$ applied by the magnetic field $b$. 
The second equation of \eqref{mhd} 
describes the evolution of the magnetic field following the so-called {\sl induction} equation.
The condition $\divv u=0$ corresponds to the incompressibility of the fluid, while the divergence-free condition on the magnetic field $b$ comes from the fact that $b$ is in the range of the ${\rm curl}$ operator.

Our aim in this paper is to study the same system in higher dimensions - i.e. $\Ri^n$, $n\ge 3$. This requires us to rewrite the system \eqref{mhd} using the exterior and interior derivative (see \cite{Mo21}) - an added benefit of this formulation is that it makes it easy to generalise the results of this paper to Riemannian manifolds.

Interpreting the scalar function $\pi$ as a $0$-form, $u$ as $1$-form and $b$ as a $2$-form, we can write \eqref{mhd} as: 
\begin{equation}
	\label{mhd1diff}\tag{MHD}
	\left\{
	\begin{array}{rclcl}
		\partial_t u+\Sh u+d \pi+ u\lrcorner\,du&=&-d^*b\lrcorner\, b&\mbox{ in }&\ouvert{0,T}\times\Ri^n\\
		\partial_t b+\Mh b&=&-d(u\lrcorner\, b)&\mbox{ in }&\ouvert{0,T}\times\Ri^n\\
		u(t,\cdot)&\in&{\rm{\sf N}}(d^*)_{|_{\Lambda^1}}&\mbox{ for all }&t\in\ouvert{0,T}\\
		b(t,\cdot)&\in&{\rm{\sf R}}(d)_{|_{\Lambda^2}}&\mbox{ for all }&t\in\ouvert{0,T},\\
	\end{array}
	\right.
\end{equation}
Where $d$ is the exterior derivative, $d^*$ is its adjoint, $\Sh$ is the Stokes operator and $\Mh$ is the Maxwell operator. We detail the signification of those notations in section \ref{tools}, but for now let us add the following remarks:
\begin{remark}
	\begin{itemize}
		\item All the terms in the first equation are 1-forms, while all the terms in the second equation are 2-forms.
		\item $\Nn(d)$ is the null of $d$ ; $\Rr(d^*)$ is the range of $d^*$.
		\item In $\Ri^3$ the magnetic field $b$ is a 2-form, but can be identified as a 1-form in $\Ri^3$ using the Hodge-star operator. This is however impossible in dimension $n$.
	\end{itemize}
\end{remark}

As for the Navier-Stokes system, the system \eqref{mhd1diff} with $T=\infty$ is invariant under the scaling
\begin{align*}
	u_\lambda(t,x)&=\lambda u(\lambda^2t,\lambda x) \\
	b_\lambda(t,x)&=\lambda b(\lambda^2t,\lambda x) \\
	\pi_\lambda(t,x)&=\lambda^2 \pi(\lambda^2t,\lambda x),
\end{align*} for $\lambda>0$, $t>0$ and $x\in\Ri^n$. This suggests two possible critical spaces for $(u,b)$: either 
\begin{equation*}
L^q\left(\demi{0,\infty};  L^p(\Ri^n, \Lambda^1)\right)\times L^q(\demi{0,\infty};  L^p(\Ri^n,  \Lambda^2)),
\end{equation*}
with $\frac{n}{p}+\frac{2}{q}=1$, or 
\begin{equation*}
\CC(\demi{0,\infty};L^n(\Ri^n,\Lambda^1))\times \CC(\demi{0,\infty};L^n(\Ri^n,\Lambda^2)).	
\end{equation*}
The purpose of this paper is to prove existence and uniqueness of mild solutions (as defined in Definition~\ref{def_mild})of the \eqref{mhd1diff} system in $\Ri^n$. 
Section \ref{LqLp} is devoted to $L^q_tL^p_x$ spaces, with Theorem \ref{thm:mhd1global} and Theorem \ref{thm:mhd1local}  proving respectively the global existence (in time) of mild solutions for small initial data and the local existence for arbitrary large initial data.

Section \ref{CtLn} and \ref{Uniqueness} are devoted to $\CC_tL^n_x$ spaces. In section \ref{CtLn} we prove the existence of mild solutions (Theorems \ref{thm:mhd1global_2} and \ref{thm:mhd1local_2}), while in Section \ref{Uniqueness}, Theorem \ref{thm:uniqueness} we prove that those mild solutions are in fact unique.

\section{Tools and notations}
\label{tools}

In this section we gathered notations and results about differential forms as well as the Laplacian, Stokes and Maxwell operators. Most of it is directly taken from \cite{Mo21} (which however focuses on bounded domains), while the proof for the different results stated can be found in \cite{McIM18} as well as \cite{MM09a} and \cite{MM09b}. 

\begin{notation}
	Let $A$ be an (unbounded) operator on a Banach space $X$. We denote by ${\rm{\sf D}}(A)$ 
	its domain, ${\rm{\sf R}}(A)$ its range and ${\rm{\sf N}}(A)$ its null space.
\end{notation}
We also denote by $\CS(\Ri^n)$ the usual Schwartz space on $\Ri^n$.
\subsection{Differential forms}

\paragraph{Exterior algebra}
We consider the exterior algebra $\Lambda=\Lambda^0\oplus\Lambda^1\oplus\dots\oplus\Lambda^n$ of
${\mathbb{R}}^n$, and we denote by $\{e_I, I\subset \llbracket 1,n \rrbracket \}$ the canonic basis of $\Lambda$, where $e_I=e_{j_1}\wedge e_{j_2}\wedge\dots\wedge e_{j_\ell}$
for $I=\{j_1,\dots,j_\ell\}$ with $\ j_1<j_2<\dots<j_{\ell}$.
\\
Note that $\Lambda^0$ is in fact $\Ri^n$, and that for $\ell<0$ or $\ell>n$ we set $\Lambda^l=\{0\}$.
\\
The basic operations on the exterior algebra $\Lambda$ are 
\begin{enumerate}[$(i)$ ]
	\item 
	the exterior product $\wedge:\Lambda^k\times\Lambda^\ell\to\Lambda^{k+\ell}$,
	\item
	the interior product $\lrcorner\,:\Lambda^k\times\Lambda^\ell\to\Lambda^{\ell-k}$,
	\item
	the inner product $\langle\cdot,\cdot\rangle:\Lambda^\ell\times\Lambda^\ell
	\to{\mathbb{R}}$. 
\end{enumerate}
These correspond to the following operations in $\Ri^3$: Let $u$ be a vector, interpreted as a 1-form:
\begin{itemize}
	\item[-]
	for $\varphi$ scalar, interpreted as a 0-form: $u\wedge \varphi=\varphi u$, $u\lrcorner\, \varphi=0$.
	\item[-]
	for $\varphi$ scalar, interpreted as a 3-form: $u\wedge v\varphi=0$, $u\lrcorner\, \varphi=\varphi u$;
	\item[-]
	$v$ vector, interpreted as a 1-form: $u\wedge v=u\times v$, $u\lrcorner\, v=u\cdot v$;
	\item[-]
	$v$ vector, interpreted as a 2-form: $u\wedge v=u\cdot v$, $u\lrcorner\, v =-u\times v$.
\end{itemize}

\paragraph{Exterior and interior derivatives}
We denote the {\sl exterior derivative} by $d:=\nabla\wedge=\sum_{j=1}^n \partial_j e_j\wedge$ and 
the {\sl interior derivative} (or co-derivative) by
$\delta:=-\nabla\lrcorner\,=-\sum_{j=1}^n \partial_j e_j\lrcorner\,$. They act on 
{\sl differential forms} from $\Ri^n$ to the exterior algebra  $\Lambda=\Lambda^0\oplus\Lambda^1\oplus\dots\oplus\Lambda^n$ of
${\mathbb{R}}^n$, and satisfy $d^2=d\circ d=0$ 
and $\delta^2=\delta\circ\delta=0$.

In $\Ri^3$ they correspond to the following operators:
\begin{align}
	d&: \Lambda^0=\Ri \overset{\nabla}{\longrightarrow} \Lambda^1=\Ri^3 \overset{\curl}{\longrightarrow} \Lambda^2=\Ri^3 \overset{\divv}{\longrightarrow} \Lambda^3=\Ri \\
	\delta&: \Lambda^0=\Ri \overset{-\divv}{\longleftarrow} \Lambda^1=\Ri^3 \overset{\curl}{\longleftarrow} \Lambda^2=\Ri^3 \overset{-\nabla}{\longleftarrow} \Lambda^3=\Ri 
\end{align}

We denote by 
${\rm{\sf D}}(d)$ the domain of (the differential operator) $d$ and by ${\rm{\sf D}}(\delta)$ the domain of $\delta$. They are defined by
\begin{equation}
{\rm{\sf D}}(d):=\bigl\{u\in L^2(\Ri^n,\Lambda); du\in L^2(\Ri^n,\Lambda)\bigr\}
\quad \mbox{and}\quad
{\rm{\sf D}}(\delta):=\bigl\{u\in L^2(\Ri^n,\Lambda); 
\delta u\in L^2(\Ri^n,\Lambda)\bigr\}.
\end{equation}
Similarly, their domains in $L^p$ are:
\begin{equation}
{\rm{\sf D}}^p(d):=\bigl\{u\in L^p(\Ri^n,\Lambda); 
du\in L^p(\Ri^n,\Lambda)\bigr\}
\  \mbox{ and }\  
{\rm{\sf D}}^p(\delta):=\bigl\{u\in L^p(\Ri^n,\Lambda); 
\delta u\in L^p(\Ri^n,\Lambda)\bigr\}.
\end{equation}

We also consider the maximal adjoint operator of $d$ in $L^2(\Ri^n,\Lambda)$, denoted by $d^*$. In $\Ri^n$, $\delta=d^*$, and we will use $d^*$ in the rest of this paper. 

For more details on $d$ and $\delta$, we refer to \cite[Section~2]{AMcI04} and 
\cite[Section~2]{CMcI10}. Both these papers also contain some historical background.

\subsection{Laplacian, Stokes and Maxwell operators}

\begin{definition}
	The {\sl Dirac operator} on $\Ri^n$ is
	\[
	\Dh:=d+d^*=d+\delta.
	\]
	The {\sl Laplacian operator} on $\Ri^n$ is defined as
	\begin{equation*}
		-\Lh :=\Dh^2=d d^*+ d^*d=d \delta+ \delta d.
	\end{equation*}
\end{definition}
\begin{remark}
	For $1$ forms in 3 dimension, this last equation correspond to the well-known identity $-\Delta = \curl \curl - \nabla\divv$.
\end{remark}
$\Dh$ is a closed densely defined operator on $L^2(\Ri^n,\Lambda)$ and we have the following Hodge decomposition (see \cite[Section 4, proof of Proposition 2.2]{AKMcI06Invent}):
\begin{align}
	\label{H2}\tag{$H_2$}
	L^2(\Ri^n,\Lambda)=&\overline{{\rm{\sf R}}(d)}\stackrel{\bot}{\oplus}
	\overline{{\rm{\sf R}}(d^*)}\stackrel{\bot}{\oplus}{\rm{\sf N}}(\Dh)\\
	=&\overline{{\rm{\sf R}}(d)}\stackrel{\bot}{\oplus}{\rm{\sf N}}(d^*)
	\label{H2Rd}\\
	=&{\rm{\sf N}}(d)\stackrel{\bot}{\oplus}\overline{{\rm{\sf R}}(d^*)}.
	\label{H2Nd}
\end{align}
Note that the harmonic forms in $L^2$ on $\Ri^d$ are trivial, so ${\rm{\sf N}}(\Dh)= {\rm{\sf N}}(d)\cap{\rm{\sf N}}(d^*)
={\rm{\sf N}}\bigl(\Lh \bigr)=\{0\}$. The orthogonal 
projection from $L^2(\Ri^n,\Lambda)$ onto ${\rm{\sf N}}(d^*)$ (see \eqref{H2Rd}),
restricted to $1$-forms, is the 
well-known Helmholtz (or Leray) projection denoted by ${\mathbb{P}}$. 

The Hodge decompositions exist also in $L^p$ (see \cite[Theorems~2.4.2 and 2.4.14]{Schw95}):
\begin{align}
	\label{Hp}\tag{$H_p$}
	L^p(\Ri^n,\Lambda)=&\overline{{\rm{\sf R}}^p(d)}\oplus
	\overline{{\rm{\sf R}}^p(d^*)}\oplus{\rm{\sf N}}(\Dh)\\
	=&\overline{{\rm{\sf R}}^p(d)} \oplus{\rm{\sf N}}^p(d^*)\label{HpRd}\\
	=&{\rm{\sf N}}^p(d) \oplus\overline{{\rm{\sf R}}^p(d^*)}\label{HpNd}
\end{align}
for all $p\in\ouvert{1,\infty}$ and the projection ${\mathbb{P}}:L^p(\Ri^n,\Lambda^1)\to 
{\rm{\sf N}}^p(d^*)_{|_{\Lambda^1}}$ extends accordingly. 

\begin{definition}
	\begin{itemize}
		\item We denote by $\Sh$ the Stokes operator:
	\begin{equation}
		\Sh:=\Dh^2=d^*d \ {\rm in} \ N^2(d*)_{|\Lambda^1},
	\end{equation}
where $N^2(d^*)_{|\Lambda^1}$ is the restriction of $N^2(d^*)$ to the space of 1-forms $\Lambda^1$.
		\item We denote by $\Mh$ the Maxwell operator:
		\begin{equation}
		\Mh:=\Dh^2=dd^* \ {\rm in} \ N^2(d)_{|\Lambda^2},
	\end{equation}
	where $N^2(d)_{|\Lambda^2}$ is the restriction of $N^2(d)$ to the space of 2-forms $\Lambda^2$.
	\end{itemize}
\end{definition}
\begin{remark}
	On $\Ri^n$, $\Sh=-\Lh_{|N^2(d*)_{|\Lambda^1}}$ and $\Mh=-\Lh_{N^2(d)_{|\Lambda^2}}.$. This means that the two following theorems, written for the Laplacian operator $\Lh$, are also true for both the Stokes and Maxwell operator.
\end{remark}

First those operators are sectorial and thus admit a bounded holomorphic functional calculus.
\begin{theorem}
	\label{thm:HodgeL&S}
	\begin{enumerate}
		\item The Laplacian operator $-\Lh $ is sectorial of angle $0$ in $L^p(\Ri^n,\Lambda)$ and for all $\mu\in \ouvert{0,\frac{\pi}{2}}$, $-\Lh $ admits a bounded $S^\circ_{\mu+}$ holomorphic functional calculus in $L^p(\Ri^n,\Lambda)$.
	\end{enumerate}
\end{theorem}

And secondly they verify the maximal regularity property, which is crucial for our proof:
\begin{theorem}[Maximal regularity]\label{Sylvie_disciple}
	Let $1<p,q<\infty$ and let R be the operator defined for $f\in L^1_{{\rm loc}}(]0,\infty[;\CS'(\Ri^n))$ by
	\begin{equation}
		Rf(t)=\int_0^t e^{(t-s)\Lh }f(s)\dd s, \quad \forall t>0.
	\end{equation}
	This operator is bounded from $L^q(\ouvert{0,\infty};L^p(\Ri^n))$ to $\dot{W}^{1,q}(\ouvert{0,\infty};L^p(\Ri^n))\cap L^q(\ouvert{0,\infty};\dot{W}^{2,p}(\Ri^n))$. 
	In particular the operator $\Lh R$ is bounded in $L^q(\ouvert{0,\infty};L^p(\Ri^n))$.
	
	Moreover there exists a constant $C_{q,p}$ such that 
	\begin{equation}\label{regmax}
		\|\frac{\dd}{\dd t}Rf\|_{L^q_tL^p_x}+\|\Lh Rf\|_{L^q_tL^p_x}+\|(-\Lh )^\alpha(\frac{\dd}{\dd t})^{1-\alpha}Rf\|_{L^q_tL^p_x} \le C_{q,p}\|f\|_{L^q_tL^p_x},
	\end{equation}
	for all $\alpha\in \ouvert{0,1}$.
\end{theorem}

The proof can be found in \cite[Chapter IV, \S 3]{LSU68}.

\subsection{The magnetohydrodynamical system}

Let us recall the magnetohydrodynamical system \eqref{mhd1diff}:
\begin{equation}
\tag{MHD}
	\left\{
	\begin{array}{rclcl}
		\partial_t u+\Sh u+d \pi+ u\lrcorner\,du&=&-d^*b\lrcorner\, b&\mbox{ in }&\ouvert{0,T}\times\Ri^n\\
		\partial_t b+\Mh b&=&-d(u\lrcorner\, b)&\mbox{ in }&\ouvert{0,T}\times\Ri^n\\
		u(t,\cdot)&\in&{\rm{\sf N}}(d^*)_{|_{\Lambda^1}}&\mbox{ for all }&t\in\ouvert{0,T}\\
		b(t,\cdot)&\in&{\rm{\sf R}}(d)_{|_{\Lambda^2}}&\mbox{ for all }&t\in\ouvert{0,T}.\\
	\end{array}
	\right.
\end{equation}

\begin{definition}\label{def_mild}
	A mild solution of the system \eqref{mhd1diff} with initial condition $u_0\in N(d^*)_{|\Lambda^1}$ and $b_0\in R(d)_{|\Lambda^2}$ is a pair $(u,b)$ such that $u$ is $1-$form on $\Ri^n$, $b$ is a $2-$form on $\Ri^n$, and $(u,b)$ satisfies
	\begin{align}
		\label{mildsolmhd1u}
		u(t)=&e^{-t\Sh}u_0+\int_0^te^{-(t-s)\Sh}{\mathbb{P}}\bigl(-u(s)\lrcorner\, du(s)\bigr)\,{\rm d}s
		+\int_0^te^{-(t-s)\Sh}{\mathbb{P}}\bigl(-d^*b(s)\lrcorner\,b(s)\bigr)\,{\rm d}s,
		\\
		=& a_1(t) + B_1(u,u)(t)+B_2(b,b)(t)\\
		\label{mildsolmhd1b}
		b(t)=&e^{-t \Mh}b_0+\int_0^te^{-(t-s) \Mh}\Bigl(-d\bigl(u(s)\lrcorner\,b(s)\bigr)\Bigr)\,{\rm d}s\\
		=&a_2(t)+B_3(u,b)(t)
	\end{align} 
for all $t\in \ouvert{0,T}$.
\end{definition}

In the formalism we use, it is easy to see that the bilinear terms $B_1$ and $B_2$ are almost identical. In fact we will focus on $B_1$ and skip the details for $B_2$ altogether. The bilinear form for the magnetic field $B_3$ is different however, and in sections \ref{CtLn} and \ref{Uniqueness} we will need the following Leibniz-style inequality:
	\begin{lemma}
	Let $\alpha$, $\beta$, $\alpha'$, $\beta'$ and $\gamma$ be such that $\frac{1}{\alpha}+\frac{1}{\beta}=\frac{1}{\alpha'}+\frac{1}{\beta'}=\frac{1}{\gamma}$.
	There exists a constant $C_p$ such that
	
	\begin{equation}\label{Leibniz}
		\|d(\omega_1\lrcorner\,\omega_2)\|_{\gamma}\le 
		C_p \bigl(\|\Dh\omega_1\|_\alpha\|\omega_2\|_\beta +\|\omega_1\|_{\alpha'}\|\Dh\omega_2\|_{\beta'}\bigr),
	\end{equation}
	for all $\omega_1\in {\rm{\sf D}}^\alpha(\Dh)\cap L^{\alpha'}(\Ri^n,\Lambda^1)$ and all
	$\omega_2\in {\rm{\sf D}}^{\beta'}(\Dh)\cap L^{\beta}(\Ri^n,\Lambda^2)$.
\end{lemma}
\begin{proof}
	On $\Ri^n$ we get $-\Delta = \Dh^2$, so $\nabla = [\nabla (-\Delta)^{-1} \Dh] \Dh = [\nabla(-\Delta)^{-1/2}] [(-\Delta)^{-1/2}\Dh] \Dh$. So $\nabla$ is controlled by $\Dh$.
\end{proof}

\begin{remark}
	This estimate is an open question for low-regularity domains and in particular for Lipschitz domains as discussed in \cite{Mo21}.
\end{remark}

\section{Existence in $L^q_tL^p_x$ spaces}\label{LqLp}

In this section we consider solutions which are $L^q$ in time and $L^p$ in space. We prove global existence of those solutions for small initial data and local existence for arbitrary large initial data.

Our proofs are based on the classical Picard fixed point theorem, 
already used for the Navier-Stokes equations by Fujita and Kato \cite{FK64}
(see also \cite{M06}) and in \cite{BM20} (see also \cite{BH20}) for the Boussinesq system. Our most recent inspiration is a paper by Monniaux \cite{Mo21} on the 3-dimensional \eqref{mhd1diff} system.
Most of the tools used here appeared in the paper \cite{MM09b}; see also \cite{McIM18}.

Let us start our main theorems:

\begin{theorem}[Global existence]
	\label{thm:mhd1global}
	Let $(p,q)$ such that $\frac{n}{p}+\frac{2}{q}=1$, $p>n$, and $q>3$.
	
	Then there exists $\varepsilon>0$ such that for all $u_0\in {\rm{\sf N}}^n(d^*)_{|_{\Lambda^1}}$ and 
	$b_0\in {\rm{\sf R}}^n(d)_{|_{\Lambda^2}}$ with 
	\begin{align}
		\|u_0\|_{\dot{B}^{-\frac{2}{q}}_{p,q}} &+ \|u_0\|_{\dot{B}^{-\frac{4}{q}+1}_{\frac{p}{2},\frac{q}{2}}}\le \e \\
		{\rm and } \ \|b_0\|_{\dot{B}^{-\frac{2}{q}}_{p,q}} &+ \|b_0\|_{\dot{B}^{-\frac{4}{q}+1}_{\frac{p}{2},\frac{q}{2}}}\le \e,
	\end{align} 
the system \eqref{mhd1diff} admits a mild solution $(u,b) \in L^q([0,\infty[;L^p(\Ri^n,\Lambda^1))\times L^q([0,\infty[;L^p(\Ri^n,\Lambda^2))$.
\end{theorem}

\begin{theorem}[Local existence]
	\label{thm:mhd1local}
	Let $(p,q)$ such that $\frac{n}{p}+\frac{2}{q}=1$, $p>n$, and $q>3$.
	
	Then for all 
	$u_0\in {\rm{\sf N}}^n(d^*)_{|_{\Lambda^1}}$ and 
	$b_0\in {\rm{\sf R}}^n(d)_{|_{\Lambda^2}}$ there exists $T>0$ such that the
	system \eqref{mhd1diff} admits a mild solution 
	$u \in L^q([0,\infty[;L^p(\Ri^n,\Lambda^1))$ and $b\in L^q([0,\infty[;L^p(\Ri^n,\Lambda^2))$.
\end{theorem}

\begin{proof}	
	For $0<T\le \infty$, let us consider the spaces 
	\begin{equation}
		\CU_T:=\left\{u\in L^q([0,T[;\Nn^p(d^*)_{|\Lambda^1}); \ du\in L^{\frac{q}{2}}([0,T[;L^{\frac{p}{2}}(\Ri^n,\, \Lambda^2))\right\}
	\end{equation}
	and
	\begin{equation}
		\CB_T:=\left\{b\in L^q([0,T[;\Rr^p(d)_{|\Lambda^2}); \ d^* b\in L^{\frac{q}{2}}([0,T[;L^{\frac{p}{2}}(\Ri^n,\, \Lambda^1))\right\}
	\end{equation}
	endowed with their natural norms
	
	\begin{align*}
		\|u\|_{\CU_T} &= \|u\|_{L^q([0,T[;L^p(\Ri^n,\Lambda^1))}+\|du\|_{L^{\frac{q}{2}}([0,T[;L^{\frac{p}{2}}(\Ri^n,\Lambda^2))} \\
		\|b\|_{\CB_T} &= \|b\|_{L^q([0,T[;L^p(\Ri^n,\Lambda^2))}+\|d^*b\|_{L^{\frac{q}{2}}([0,T[;L^{\frac{p}{2}}(\Ri^n,\Lambda^1))}.
	\end{align*}

The proof relies on the Picard fixed-point theorem (see \cite[Theorem 15.1]{Lem02}): the system
	\begin{align}
		\label{eq:fixedpoint}
		u=a_1+B_1(u,u)+B_2(b,b)\quad\mbox{and}\quad b=a_2+B_3(u,b), \quad (u,b)\in{\mathscr{U}}_T
	\end{align}
	can be reformulated as
	\begin{equation}
		\label{eq:picard1}
		{\boldsymbol{U}}={\boldsymbol{a}}+{\boldsymbol{\cb}}({\boldsymbol{U}},{\boldsymbol{U}})
	\end{equation}
	where ${\boldsymbol{U}}=(u,b)\in {\mathscr{U}}_T\times{\mathscr{B}}_T$, ${\boldsymbol{a}}=(a_1,a_2)$ and 
	${\boldsymbol{\cb}}({\boldsymbol{U}},{\boldsymbol{U'}})=(B_1(u,u')+B_2(b,b'),B_3(u,b'))$ if 
	${\boldsymbol{U}}=(u,b)$ and ${\boldsymbol{U'}}=(u',b')$. On ${\mathscr{U}}_T\times{\mathscr{B}}_T$
	we choose the norm 
	$\|(u,b)\|_{{\mathscr{U}}_T\times{\mathscr{B}}_T}:=\|u\|_{{\mathscr{U}}_T}+\|b\|_{{\mathscr{B}}_T}$.
	
We split the proof into two lemmas: Lemma~\ref{lem:initcond1} concerns the linear part, while Lemma~\ref{lem_Bilin1} concerns the bilinear operator $\boldsymbol{\cb}$.
	\begin{lemma}\label{lem:initcond1}
		For $u_0\in \dot{B}^{-\frac{2}{q}}_{p,q}(\Ri^n,\Lambda^1)\cap \dot{B}^{-\frac{4}{q}+1}_{\frac{p}{2},\frac{q}{2}}(\Ri^n,\Lambda^1)$ with $d^* u_0=0$ and $b_0\in \dot{B}^{-\frac{2}{q}}_{p,q}(\Ri^n,\Lambda^2)\cap\dot{B}^{-\frac{4}{q}+1}_{\frac{p}{2},\frac{q}{2}}(\Ri^n,\Lambda^2)$ with $d b_0=0$, then 
		\begin{enumerate}
			\item $a_1: \, t\mapsto e^{-t\Sh }u_0 \in \CU_T$
			\item $a_2: \, t\mapsto e^{-t\Mh}b_0 \in \CB_T$,
		\end{enumerate}
		for all $T\in ]0,+\infty]$. Besides for all $\varepsilon>0$, there exists $T>0$ such that 
		\begin{equation}\label{a1a2_le_epsilon}
			\|a_1\|_{\CU_T}+\|a_2\|_{\CB_T} \le \varepsilon
		\end{equation} 
	\end{lemma}
	
	\begin{lemma}\label{lem_Bilin1}
		The bilinear operators $B_1$, $B_2$ and $B_3$ are bounded in the following spaces:
		\begin{enumerate}
			\item
			$B_1:\CU_T\times\CU_T \rightarrow \CU_T$, \label{B_1}
			\item
			$B_2:\CB_T\times \CB_T \rightarrow \CU_T$, \label{B_2}
			\item
			$B_3:\CU_T\times \CB_T\rightarrow \CB_T$ \label{B_3}
		\end{enumerate}
		with norms independent from $T>0$.
	\end{lemma}
	
The boundedness of the operator ${\boldsymbol{\cb}}$ is now obvious: let $\boldsymbol{U}=(u,b)\in \CU_T\times\CB_T$ and $\boldsymbol{U'}=(u',b')\in \CU_T\times \CB_T$. Then
\begin{align*}
	\left\| {\boldsymbol{\cb}}({\boldsymbol{U}},{\boldsymbol{U'}}) \right\|_{\CU_T\times\CB_T} &= \left\| B_1(u,u')+B_2(b,b')\right\|_{\CU_T}
	+ \left\| B_3(u,b')\right\|_{\CB_T} \\
	&\le K\left( \|u\|_{\CU_T} \|u'\|_{\CU_T} + \|b\|_{\CB_T}\|b'\|_{\CB_T} + \|u\|_{\CU_T}\|b'\|_{\CB_T} \right) \\
	&\le K \|\boldsymbol{U}\|_{\CU_T\times\CB_T} \|\boldsymbol{U'}\|_{\CU_T\times\CB_T},
\end{align*}
	where $K$ is a constant independent from $T>0$. 
	
	Let then $\varepsilon=\frac{1}{4K}$. By Lemma~\ref{lem:initcond1}, 
	for $u_0\in {\rm{\sf N}}^n(d^*)_{|_{\Lambda^1}}$, and
	$b_0\in {\rm{\sf R}}^n(d)_{|_{\Lambda^2}}$, there exists $T\le \infty$ such that 
	$\|a_1\|_{\CU_T}+\|a_2\|_{\CB_T} \le \varepsilon$ holds for $\varepsilon=\frac{1}{4K}$.
	Then by Picard's fixed point theorem the system \eqref{eq:picard1} 
	admits a unique solution ${\boldsymbol{U}}=(u,b)\in {\mathscr{U}}_T\times{\mathscr{B}}_T$.
\end{proof}
	
\begin{proof}[Lemma~\ref{lem:initcond1}]
	Let $\e>0$. Let $u_0\in \dot{B}^{-\frac{2}{q}}_{p,q}(\Ri^n,\Lambda^1)\cap \dot{B}^{-\frac{4}{q}+1}_{\frac{p}{2},\frac{q}{2}}(\Ri^n,\Lambda^1)$ with $d^* u_0=0$ and $b_0\in \dot{B}^{-\frac{2}{q}}_{p,q}(\Ri^n,\Lambda^2)\cap\dot{B}^{-\frac{4}{q}+1}_{\frac{p}{2},\frac{q}{2}}(\Ri^n,\Lambda^2)$ with $d b_0=0$.
\begin{enumerate}
\item First we prove that the semigroups $t\mapsto a_1(t)=e^{-t\Sh}u_0$ and $t\mapsto a_2(t)=e^{-t\Mh}b_0$ are respectively in $\CU_T$ and $\CB_T$.
	
Let $T=+\infty$. Thanks to \cite[Lemma~2.34]{BCD11} we have
\begin{align*}
	\|u_0\|_{\dot{B}^{-\frac{2}{q}}_{p,q}} &\sim_{p,q} \left\|t\mapsto \|t^\frac{1}{q}e^{t\Lh}u_0\|_{L^p_x}\right\|_{L^q(\demi{0,\infty},\frac{\dd t}{t})} \\
	&\sim_{p,q} \left(\int_{0}^{+\infty}\|e^{t\Delta}u_0\|_{L^p_x}^q\dd t\right)^\frac{1}{q}\\
	&\sim_{p,q} \|t\mapsto e^{-t\Sh}u_0\|_{L^q_tL^p_x}.
\end{align*}
Now, using the fact that $(-\Lh)^\frac{1}{2}: \ \dot{B}^s_{\frac{p}{2},\frac{q}{2}} \rightarrow \dot{B}^{s-1}_{\frac{p}{2},\frac{q}{2}}$ is an isomorphism we get
\begin{equation*}
	\|u_0\|_{\dot{B}^{-\frac{4}{q}+1}_{\frac{p}{2},\frac{q}{2}}} \sim_{p,q} \|(-\Lh)^\frac{1}{2}u_0\|_{\dot{B}^{-\frac{4}{q}}_{\frac{p}{2},\frac{q}{2}}}.
\end{equation*}
Then using \cite[Lemma~2.34]{BCD11} again we get 
\begin{align*}
	\|u_0\|_{\dot{B}^{-\frac{4}{q}+1}_{\frac{p}{2},\frac{q}{2}}} &\sim_{p,q} \left\|t\mapsto \|t^\frac{q}{2}(-\Lh)^{\frac{1}{2}}e^{t\Lh}u_0\|_{L^p_x}\right\|_{L^q(\demi{0,\infty},\frac{\dd t}{t})} \\
	&\sim_{p,q} \| t\mapsto (-\Lh)^\frac{1}{2}e^{t\Lh}u_0 \|_{L^\frac{q}{2}_t L^\frac{p}{2}_x} \\
	&\sim_{p,q} \| de^{-t\Sh}u_0 \|_{L^\frac{q}{2}_t L^\frac{p}{2}_x},
\end{align*}
where the last line comes from the fact that $d^*u_0=0$ and $\| \Dh \cdot \|_\frac{p}{2} \sim \| (-\Lh)^\frac{1}{2}\|_\frac{p}{2}$.

Hence for $T\in \Ri^+$,
\begin{equation}
	\|a_1\|_{\CU_T}\le \|a_1\|_{\CU_{\infty}}\less_{p,q} \|u_0\|_{\dot{B}^{-\frac{q}{2}}_{p,q}} + \|u_0\|_{\dot{B}^{-\frac{4}{q}+1}_{\frac{p}{2},\frac{q}{2}}}.
\end{equation}
The estimate for $a_2$ is proven in a similar way: 
\begin{equation}
	\|a_2\|_{\CU_T} \le \|a_2\|_{\CU_\infty} \less_{p,q} \|b_0\|_{\dot{B}^{-\frac{q}{2}}_{p,q}} + \|b_0\|_{\dot{B}^{-\frac{4}{q}+1}_{\frac{p}{2},\frac{q}{2}}}.
\end{equation}

\item Let $\e>0$. If $u_0$ and $b_0$ have norms smaller than $\|u_0\|_{\dot{B}^{-\frac{q}{2}}_{p,q}} + \|u_0\|_{\dot{B}^{-\frac{4}{q}+1}_{\frac{p}{2},\frac{q}{2}}}$ and $\|b_0\|_{\dot{B}^{-\frac{q}{2}}_{p,q}} + \|b_0\|_{\dot{B}^{-\frac{4}{q}+1}_{\frac{p}{2},\frac{q}{2}}}$ respectively, then we  immediately have $\|a_1\|_{\CU_T}+\|a_2\|_{\CB_T} \le K_{p,q}\, \varepsilon$ for all $T\in [0,\infty]$.

Else let $T\in \Ri^+$ and let $u_0^\e\in \CS(\Ri^n,\Lambda^1)$ and $b_0^\e \in \CS(\Ri^n,\Lambda^2)$ be such that $d^*u_0=0$, $db_0=0$, and 
\begin{align*}
	\|u_0-u_0^\e\|_{\dot{B}^{-\frac{4}{q}+1}_{\frac{p}{2},\frac{q}{2}}} &\le \e \; {\rm and } \; \|u_0-u_0^\e\|_{\dot{B}^{-\frac{2}{q}}_{p,q}} \le \e, \\
	\|b_0-b_0^\e\|_{\dot{B}^{-\frac{4}{q}+1}_{\frac{p}{2},\frac{q}{2}}} &\le \e \; {\rm and } \;	\|b_0-b_0^\e\|_{\dot{B}^{-\frac{2}{q}}_{p,q}} \le \e.
\end{align*}
Let us denote for $t\in [0,T]$ $a_1^\e(t) = e^{-t\Sh}u_0^\e$ and  $a_2^\e(t) = e^{-t\Mh}b_0^\e$, and write
\begin{equation*}
	\|a_1\|_{\CU_T}\le \|a_1-a_1^\e\|_{\CU_T}+\|a_1^\e\|_{\CU_T}
	\le K_{p,q}\, \e + \|a_1^\e\|_{\CU_T}.
\end{equation*}
By definition $\|a_1^\e\|_{\CU_T}=\|a_1^\e\|_{L^q_tL^p_x}+\|d a_1^\e\|_{L^\frac{q}{2}_t L^\frac{p}{2}_x}$.
Let us consider $\| a_1^\e\|_{L^q_tL^p_x}$ first: let $s\in \Ri$ be such that $1-qs>0$. Then we get
\begin{align*}
	\| a_1^\e\|_{L^q_tL^p_x} &\le \| t\mapsto t^s\|_{L^\infty([0,T])} \left\|t\mapsto t^{-s}\|e^{-t\Sh}u_0^\e\|_{L^p_x}\right\|_{L^q([0,T])} \\
	&\less_{p,q} T^s \left\| t \mapsto \|t^{\frac{1-sq}{q}}e^{t\Lh}u_0^\e\|_{L^p_x}\right\|_{L^q([0,T],\frac{\dd t}{t})} \\
	&\less_{p,q} T^s \|u_0^\e\|_{\dot{B}^{-2\frac{1-sq}{q}}_{p,q}}.
\end{align*}
Similarly let $s\in \Ri$ be such that $1-\frac{sq}{2}>0$.
Then 
\begin{align*}
	\|d a_1^\e \|_{L^\frac{q}{2}_t L^{\frac{p}{2}}_x} &\less_{p,q} \left\|t\mapsto e^{t\Lh}(-\Lh)^\frac{1}{2}u_0^\e \right\|_{L^\frac{q}{2}_t L^{\frac{p}{2}}_x} \\
	&\less_{p,q} T^s \| (-\Lh)^\frac{1}{2}u_0^\e \|_{\dot{B}^{-2\frac{2-sq}{q}}_{\frac{p}{2},\frac{q}{2}}} \\
	&\less_{p,q} T^s \| u_0^\e \|_{\dot{B}^{-2\frac{2-sq}{q}+1}_{\frac{p}{2},\frac{q}{2}}}.
\end{align*}
Since $u_0^\e \in \CS(\Ri^n,\Lambda^1)$, both $\| u_0^\e \|_{\dot{B}^{-2\frac{2-sq}{q}+1}_{\frac{p}{2},\frac{q}{2}}}$ and $\|u_0^\e\|_{\dot{B}^{-2\frac{1-sq}{q}}_{p,q}}$ are well-defined and finite, although they can be arbitrarily large depending on $u_0$. However taking $T$ small enough we get 
\begin{equation*}
	\| a_1^\e \|_{\CU_T} \le K_{p,q} \, \e.
\end{equation*}
And in a similar way we can prove that 
\begin{equation*}
	\| a_2^\e \|_{\CB_T} \le K_{p,q} \, \e,
\end{equation*}
which concludes our proof.
\end{enumerate}
\end{proof}

\begin{proof}[Lemma~\ref{lem_Bilin1}]
	Recall the relations on $n$, $p$ and $q$:
	\begin{equation*}
		n<p, \quad	3<q \quad {\rm and} \quad \frac{n}{p}+\frac{2}{q}=1 
	\end{equation*}
	\begin{enumerate}
	\item Recall that $B_1(u,v)(t)=\int_0^t e^{-(t-s)\Sh} \Leray(u(s)\contr dv(s))\dd s$. 
	\begin{itemize}
		\item 
		Let $\theta = \frac{n}{p}$. Then $\frac{3}{p}-\frac{2\theta}{n}=\frac{1}{p}$, so the Sobolev injection $W^{2\theta,\frac{p}{3}}\hookrightarrow L^p$ holds.
		
		For almost every $t>0$, we compute the norm in $L^p_x$ of $B_1(u,v)(t)$ in the following way:
		\begin{align*}
			\left\|B_1(u,v)(t)\right\|_{L^p_x}&=\left\| \int_0^t \Sh ^\theta e^{-(t-s)\Sh } \Sh ^{-\theta}\Leray \left(u(s)\contr dv(s)\right)\dd s \right\|_{L^p_x} \\
		(1) \qquad	&\lesssim \int_0^t \left\| \Sh ^\theta e^{-(t-s)\Sh }\right\|_{L^p\rightarrow L^p} \left\|\Sh ^{-\theta}\Leray \big(u(s)\contr dv(s)\big) \right\|_{L^p_x} \dd s   \\
		(2) \qquad	&\lesssim \int^t_0 (t-s)^{-\theta} \left\|\Sh ^{-\theta}\Leray \big(u(s)\contr dv(s)\big) \right\|_{W^{2\theta,\frac{p}{3}}_x} \dd s \\
		(3) \qquad	&\lesssim \int^t_0 (t-s)^{-\theta} \left\|\Leray \big(u(s)\contr dv(s)\big) \right\|_{L^\frac{p}{3}_x} \dd s \\
		(4)	\qquad &\lesssim \int^t_0 (t-s)^{-\theta} \left\| u(s)\contr dv(s) \right\|_{L^\frac{p}{3}_x} \dd s \\
		(5) \qquad	&\lesssim \int^t_0 (t-s)^{-\theta} \|u(s)\|_{L^p_x}\| dv(s)\|_{L^\frac{p}{2}_x}  \dd s,
		\end{align*}
	 	where (1) uses the operator norm of $\Sh^\theta e^{-(t-s)\Sh}$, (2) uses the Sobolev injection $W^{2\theta,\frac{p}{3}}\hookrightarrow L^p$, (3) uses the continuity of $\Sh^{-\theta}$ from $L^\frac{p}{3}_x$ to $W^{2\theta,\frac{p}{3}}$, (4) uses the continuity of the Leray projector $\Leray$ on $L^\frac{p}{3}_x$ and finally (5) is simply Hölder's inequality.
	 
		Since $s\mapsto s^{-\theta}=s^{-\frac{n}{p}}$ is in $L^{\frac{p}{n},\infty}$ (see \cite[Definition~1.1.5]{Gr08}) and $s\mapsto  \|u(s)\|_{L^p_x}\| dv(s)\|_{L^\frac{p}{2}_x}$ is in $L^\frac{q}{3}_t$ by Hölder's inequality, the convolution inequality $\|f\star g\|_{L^q}\less_{n,p,q} \|f\|_{L^{\frac{p}{n},\infty}}\|g\|_{L^\frac{q}{3}}$ (see \cite[Theorem~1.2.13
		]{Gr08}) yields 
		\begin{equation}
			\left\|B_1(u,v)\right\|_{L^q_t L^p_x}\less_{n,p,q} \, \| u\|_{\CU_T} \|v\|_{\CU_T}.
		\end{equation}
		
	\item	We now compute the norm of $dB_1(u,v)$. 
		Let $\theta=\frac{n}{2p}$ be such that $\frac{3}{p}-\frac{2\theta}{n}=\frac{2}{p}$, so that the Sobolev injection $W^{2\theta, \frac{p}{3}}\hookrightarrow L^\frac{p}{2}$ holds.
		Then, following the same steps as for $\| B_1(u,v)(t)\|_{L^p_x}$, we get:
		\begin{align*}
			\left\|dB_1(u,v)(t)\right\|_{L^\frac{p}{2}_x}&=\left\| \int_0^t d \Sh ^\theta  e^{-(t-s)\Sh } \Sh ^{-\theta}\Leray \left(u(s)\contr dv(s)\right)\dd s \right\|_{L^\frac{p}{2}_x} \\
	(1) \qquad		&\less \int_0^t \left\| d \Sh ^\theta e^{-(t-s)\Sh }\right\|_{L^\frac{p}{2}\rightarrow L^\frac{p}{2}} \left\|\Sh ^{-\theta}\Leray \big(u(s)\contr dv(s)\big) \right\|_{L^\frac{p}{2}_x} \dd s   \\
	(2) \qquad		&\less \int^t_0 (t-s)^{-\theta-\frac{1}{2}} \left\|\Sh ^{-\theta}\Leray \big(u(s)\contr dv(s)\big) \right\|_{W^{2\theta,\frac{p}{3}}_x} \dd s \\
	(3) \qquad		&\less \int^t_0 (t-s)^{-\theta-\frac{1}{2}} \left\|\Leray \big(u(s)\contr dv(s)\big) \right\|_{L^\frac{p}{3}_x} \dd s \\
	(4) \qquad		&\less \int^t_0 (t-s)^{-\theta-\frac{1}{2}} \left\|u(s)\contr dv(s) \right\|_{L^\frac{p}{3}_x} \dd s \\
	(5) \qquad		&\less \int^t_0 (t-s)^{-\frac{n+p}{2p}} \|u(s)\|_{L^p_x}\| dv(s)\|_{L^\frac{p}{2}_x}  \dd s,
		\end{align*}
	where (1) uses the operator norm of $\Sh^\theta e^{-(t-s)\Sh}$, (2) uses the Sobolev injection $W^{2\theta,\frac{p}{3}}\hookrightarrow L^\frac{p}{2}$, (3) uses the continuity of $\Sh^{-\theta}$ from $L^\frac{p}{3}_x$ to $W^{2\theta,\frac{p}{3}}$, (4) uses the continuity of the Leray projector $\Leray$ on $L^\frac{p}{3}_x$ and finally (5) is simply (again!) Hölder's inequality.

	Since $s\mapsto s^{-\frac{n+p}{2p}}$ is in $L^{\frac{2p}{n+p},\infty}$ (see \cite[Definition~1.1.5]{Gr08}) and $s\mapsto  \|u(s)\|_{L^p_x}\| dv(s)\|_{L^\frac{p}{2}_x}$ is in $L^\frac{q}{3}_t$ by Hölder's inequality, the convolution inequality $\|f\star g\|_{L^\frac{q}{2}}\less_{n,p,q} \|f\|_{L^{\frac{2p}{n+p},\infty}}\|g\|_{L^\frac{q}{3}}$ (see \cite[Theorem~1.4.24]{Gr08}) yields  
		\begin{equation}
			\left\|dB_1(u,v)\right\|_{L^q_t L^p_x}\lesssim \| u\|_{\CU_T} \|v\|_{\CU_T}.
		\end{equation}
		And combining both estimates yields 
		\begin{equation}
			\left\|B_1(u,v)\right\|_{\CU_T}\lesssim \| u\|_{\CU_T} \|v\|_{\CU_T}.
		\end{equation}
	\end{itemize}
	\item The boundedness of $B_2:\CB_T\times \CB_T \rightarrow \CU_T$ is proved in the exact same way.
	\item 
	\begin{itemize}
		\item The estimates on $B_3(u,b)$ is obtained in a similar way: taking $\theta=\frac{n}{2p}$ we get
		\begin{align*}
			\|B_3(u,b)(t)\|_{L^p_x}&=\left\| \int_0^t  \Mh^\theta e^{-(t-s) \Mh}\Mh^{-\theta}\Big(-d\big(u(s)\contr b(s)\big)\Big)\dd s \right\|_{L^p_x} \\
(1) \qquad	&=\left\| \int_0^t  d \Sh^\theta e^{-(t-s) \Sh}\Sh^{-\theta} \big(u(s)\contr b(s)\big)\Big)\dd s \right\|_{L^p_x} \\
(2) \qquad	&\lesssim \int^t_0 \left\|d\Sh^\theta e^{-(t-s)\Sh}\right\|_{L^p\rightarrow L^p} \|\Sh^{-\theta}u(s)\contr b(s)\|_{L^p_x} \dd s   \\
(3) \qquad	&\lesssim \int^t_0 (t-s)^{-\theta-\frac{1}{2}} \left\| \Sh^{-\theta} u(s)\contr b(s)\right\|_{W^{2\theta,\frac{p}{2}}_x} \dd s. \\
(4) \qquad	&\lesssim \int^t_0 (t-s)^{-\frac{p+n}{2p}} \|u(s)\contr b(s)\|_{L^\frac{p}{2}_x}  \dd s. \\
(5) \qquad	&\lesssim \int^t_0 (t-s)^{-\frac{p+n}{2p}} \|u(s)\|_{L^p_x} 		\|b(s)\|_{L^p_x} \dd s,
		\end{align*}
	where (1) uses the fact that $\Mh d=d\Sh$, (2) uses the operator norm of $d\Sh^\theta e^{-(t-s)\Sh}$, (3) uses the Sobolev injection $W^{2\theta,\frac{p}{2}}\hookrightarrow L^p$, (4) uses the continuity of $\Sh^{-\theta}$ from $L^\frac{p}{2}_x$ to $W^{2\theta,\frac{p}{2}}$, and (5) is again Hölder's inequality.
	
	Then as before \cite[Theorem~1.4.24]{Gr08} gives us
		\begin{equation}
			\left\|B_3(u,b)\right\|_{L^q_t L^p_x}\lesssim \| u\|_{\CU_T} \|b\|_{\CB_T}.
		\end{equation}
		
		\item To compute $\|d^* B_3(u,b)(t)\|_\frac{p}{2}$, we can then apply the maximal regularity theorem: 
		\begin{align*}
			\|d^* B_3(u,b)\|_{L^\frac{q}{2}_t L^\frac{p}{2}_x} &= \left\| t\mapsto \int_0^t d^*e^{-(t-s)\Mh}\Big(-d\big(u(s)\contr b(s)\big)\Big)\dd s \right\|_{L^\frac{q}{2}_t L^\frac{p}{2}_x}\\
		(1)\qquad	&=\left\| t\mapsto \int_0^t \Sh e^{-(t-s)\Sh}\big(-u(s)\contr b(s)\big)\dd s \right\|_{L^\frac{q}{2}_t L^\frac{p}{2}_x}\\
		(2)\qquad	&\lesssim \| u\contr b\|_{L^\frac{q}{2}_t L^\frac{p}{2}_x} \\
		(3)\qquad	&\lesssim \|u\|_{L^q_t L^p_x} \|b\|_{L^q_t L^p_x} \\
		&\lesssim \| u\|_{\CU_T} \|b\|_{\CB_T},
		\end{align*}
	where (1) uses $\Mh d=d\Sh$, (2) is Theorem~\ref{Sylvie_disciple} applied to $\Sh$, and (3) is Hölder's inequality.
		\end{itemize}
	\end{enumerate}
	This last estimate concludes our proof.
\end{proof}

\section{Existence in $\CC_tL^p_x$ spaces}\label{CtLn}

In this section we prove the global and local existence of mild solutions for the magnetohydrodynamic system \eqref{mhd1diff}, following closely \cite{Mo21}. The method is roughly the same as in the last section - using Picard's fixed point theorem and maximal regularity - with the main difference being that we rely heavily on the Leibnitz estimate \eqref{Leibniz}, which makes it difficult to generalize our results to low-regularity domains. However, contrary to the $L^q_tL^p_x$ case, we were able to prove the uniqueness of mild solutions of the system \eqref{mhd1diff} in Section \ref{Uniqueness}.

Let us start by stating our two theorems:
\begin{theorem}[Global existence]
	\label{thm:mhd1global_2}
	There exists $\varepsilon>0$ such that for all $u_0\in {\rm{\sf N}}^n(d^*)_{|_{\Lambda^1}}$ and 
	$b_0\in {\rm{\sf R}}^n(d)_{|_{\Lambda^2}}$ with $\|u_0\|_{L^n(\Ri^n,\Lambda^1)}+\|b_0\|_{L^n(\Ri^n,\Lambda^2)}\le \varepsilon$, the
	system \eqref{mhd1diff} admits a 
	mild solution 
	$u \in \CC([0,\infty[;L^n(\Ri^n,\Lambda^1))$ and $b\in \CC([0,\infty[;L^n(\Ri^n,\Lambda^2))$.
\end{theorem}

\begin{theorem}[Local existence]
	\label{thm:mhd1local_2}
	For all 
	$u_0\in {\rm{\sf N}}^n(d^*)_{|_{\Lambda^1}}$ and 
	$b_0\in {\rm{\sf R}}^n(d)_{|_{\Lambda^2}}$ there exists $T>0$ such that the
	system \eqref{mhd1diff} admits a mild solution 
	$u \in \CC(\demi{0,T};L^n(\Ri^n,\Lambda^1))$ and $b\in \CC(\demi{0,T};L^n(\Ri^n,\Lambda^2))$.
\end{theorem}
	
Let $p\in \ouvert{n,2n}$ and $\alpha=1-\frac{n}{p}$, and define the following Banach spaces for $0<T\le +\infty$:
	
	\begin{align}
		\label{eq:UT}
		\CU_T:=&\bigl\{u\in{\mathscr{C}}(\ouvert{0,T};{\rm{\sf N}}^p(d^*)_{|_{\Lambda^1}});
		du\in{\mathscr{C}}(\ouvert{0,T};L^p(\Ri^n,\Lambda^2)) :\\
		\nonumber
		&\sup_{0<t<T}\bigl(t^{\frac{\alpha}{2}}\|u(t)\|_{L^p_x}+t^{\frac{1+\alpha}{2}}\|du(t)\|_{L^p_x}\bigr)<\infty\bigr\},
	\end{align}
	endowed with the norm
		\begin{equation}
		\label{eq:normUT}
		\|u\|_{\CU_T}:=\sup_{0<t<T}\bigl(t^{\frac{\alpha}{2}}\|u(t)\|_{L^p_x}
		+t^{\frac{1+\alpha}{2}}\|du(t)\|_{L^p_x}\bigr),
	\end{equation}
	and
	\begin{align}
		\label{eq:BT}
		\CB_T:=&\bigl\{b\in{\mathscr{C}}(\ouvert{0,T};{\rm{\sf R}}^p(d)_{|_{\Lambda^2}}) ;
		d^*b\in{\mathscr{C}}(]0,T[;L^p(\Ri^n,\Lambda^1)) :\\
		\nonumber
		&\sup_{0<t<T}\bigl(t^{\frac{\alpha}{2}}\|b(t)\|_{L^p_x}+t^{\frac{1+\alpha}{2}}\|d^*b(t)\|_{L^p_x}\bigr)<\infty\bigr\},
	\end{align}
	endowed with the norm
	\begin{equation}
		\label{eq:normBT}
		\|b\|_{\CB_T}:=\sup_{0<t<T}\bigl(t^{\frac{\alpha}{2}}\|b(t)\|_{L^p_x}
		+t^{\frac{1+\alpha}{2}}\|d^*b(t)\|_{L^p_x}\bigr).
	\end{equation}

	We split the proof into three lemmas: in Lemma \ref{lem:initcond2} we study the action of the Stokes and Maxwell semi-group on initial data $u_0\in \Nn^n(d^*)_{|\Lambda^1}$ and $b_0\in \Rr^n(d)_{|\Lambda^2}$.
	In Lemma \ref{lem_Bilin2}, we prove bilinear estimates on $B_1$, $B_2$ and $B_3$, and in Lemma \ref{lem:UTBTmild} we show that solutions from the working spaces $\CU_T$ and $\CB_T$ are in fact continuous on $\demi{0,T}$ and in $L^n$ in space. 
	
	\begin{lemma}
	\label{lem:initcond2}
	For $u_0\in {\rm{\sf N}}^n(d^*)_{|_{\Lambda^1}}$ and 
	$b_0\in {\rm{\sf R}}^n(d)_{|_{\Lambda^2}}$, we have
	\begin{enumerate}
		\item
		$a_1:t\mapsto e^{-t\Sh}u_0\in\CU_T$,
		\item
		$a_2:t\mapsto e^{-t\Mh}b_0\in\CB_T$,
	\end{enumerate}
	for all $T>0$
	Moreover, for all $\varepsilon>0$, there exists $T>0$ such that
	\begin{equation}
		\label{eq:a1a2}
		\|a_1\|_{\CU_T}+\|a_2\|_{\CB_T}\le \varepsilon.
	\end{equation}
\end{lemma}

As in Section \ref{LqLp} the second lemma gives us estimates for the bilinear operator:
	\begin{lemma}\label{lem_Bilin2}
	The bilinear operators $B_1$, $B_2$ and $B_3$ are bounded in the following
	spaces:
	\begin{enumerate}
			\item
			$B_1:\CU_T\times\CU_T \rightarrow \CU_T$,
			\item
			$B_2:\CB_T\times \CB_T \rightarrow \CU_T$,
			\item
			$D:\CU_T\times \CB_T\rightarrow \CB_T$
		\end{enumerate}
		with norms independent from $T>0$.
	\end{lemma}

Our last Lemma gives us additional regularity for mild solutions of \eqref{mhd1diff}:
\begin{lemma}
	\label{lem:UTBTmild}
	Let $T>0$.
	Assume that $(u,b)\in {\mathscr{U}}_T\times{\mathscr{B}}_T$ is a mild solution of
	\eqref{mhd1diff} with initial conditions 
	$u_0\in {\rm{\sf N}}^n(d^*)_{|_{\Lambda^1}}$ and 
	$b_0\in {\rm{\sf R}}^n(d)_{|_{\Lambda^2}}$. Then
	$u\in {\mathscr{C}}_b([0,T[;{\rm{\sf N}}^n(d^*)_{|_{\Lambda^1}})$ and
	$b\in {\mathscr{C}}_b([0,T[;{\rm{\sf R}}^n(d)_{|_{\Lambda^2}})$.
\end{lemma}

\begin{proof}[Proof of Theorems~\ref{thm:mhd1global_2} and \ref{thm:mhd1local_2}]
	The system
	\begin{align}
		\label{eq:fixedpoint}
		u=a_1+B_1(u,u)+B_2(b,b)\quad\mbox{and}\quad b=a_2+B_3(u,b), \quad (u,b)\in{\mathscr{U}}_T
	\end{align}
	can be reformulated as
	\begin{equation}
		\label{eq:picard}
		{\boldsymbol{u}}={\boldsymbol{a}}+{\boldsymbol{\cb}}({\boldsymbol{u}},{\boldsymbol{u}})
	\end{equation}
	where ${\boldsymbol{u}}=(u,b)\in {\mathscr{U}}_T\times{\mathscr{B}}_T$, ${\boldsymbol{a}}=(a_1,a_2)$ and 
	${\boldsymbol{B}}({\boldsymbol{u}},{\boldsymbol{v}})=(B_1(u,v)+B_2(b,b'),B_3(u,b'))$ if 
	${\boldsymbol{u}}=(u,b)$ and ${\boldsymbol{v}}=(v,b')$. On ${\mathscr{U}}_T\times{\mathscr{B}}_T$
	we choose the norm 
	$\|(u,b)\|_{{\mathscr{U}}_T\times{\mathscr{B}}_T}:=\|u\|_{{\mathscr{U}}_T}+\|b\|_{{\mathscr{B}}_T}$.
	As in section \ref{LqLp}, one can easily check, using Lemma~\ref{lem_Bilin2}, that
	\[
	\|{\boldsymbol{\cb}}({\boldsymbol{u}},{\boldsymbol{v}})\|_{{\mathscr{U}}_T\times{\mathscr{B}}_T}
	\le C \|{\boldsymbol{u}}\|_{{\mathscr{U}}_T\times{\mathscr{B}}_T} 
	\|{\boldsymbol{v}}\|_{{\mathscr{U}}_T\times{\mathscr{B}}_T}
	\]
	where $C$ is a constant independent from $T>0$. We can then apply Picard's fixed point theorem
	to prove that for $u_0\in {\rm{\sf N}}^n(d^*)_{|_{\Lambda^1}}$ and 
	$b_0\in {\rm{\sf R}}^n(d)_{|_{\Lambda^2}}$, with $T\le \infty$ such that 
	\eqref{eq:a1a2} holds for $\varepsilon=\frac{1}{4C}$, the system \eqref{eq:picard} 
	admits a unique solution ${\boldsymbol{u}}=(u,b)\in {\mathscr{U}}_T\times{\mathscr{B}}_T$.
	By Lemma~\ref{lem:UTBTmild}, this provides a mild solution 
	$(u,b)\in{\mathscr{C}}_b(\demi{0,T};{\rm{\sf N}}^n(d^*)_{|_{\Lambda^1}})\times
	{\mathscr{C}}_b(\demi{0,T};{\rm{\sf R}}^n(d)_{|_{\Lambda^2}})$
	of \eqref{mhd1diff}.
\end{proof}

\begin{proof}[Proof of Lemma~\ref{lem:initcond2}]
	Let $\varepsilon>0$ and let $u_0\in \Nn^n(d^*)_{\Lambda^1}$ and $b_0\in \Rr^n(d)_{\Lambda^2}$.
	By Theorem \ref{thm:HodgeL&S}, the semi-group $e^{-t\Sh}$ and $e^{-t\Mh}$ are bounded and there exists constants $c_{\alpha,p}^S$ and $c_{\alpha,p}^M$ such that for all $T>0$
	\begin{equation}\label{eq:données_petites}
		\| a_1\|_{\CU_T}+\|a_2\|_{\CB_T}\le c_{\alpha,p}^S\|u_0\|_{L^p_x} +c_{\alpha,p}^M\|b_0\|_{L^p_x}. 
	\end{equation}
	Hence if $\|u_0\|_{L^p_x}$ and $\|b_0\|_{L^p_x}$ are small enough, the inequality \eqref{eq:a1a2} $\|a_1\|_{\CU_T}+\|a_2\|_{\CB_T}\le \varepsilon$ holds.

	For any $u_0\in \Nn^n(d^*)_{\Lambda^1}$ and $b_0 \in \Rr^n(d)_{\Lambda^2} $, with arbitrary norms, let $u_0^\e\in \Nn^p(d^*)_{\Lambda^1}$ and $b_0^\e \in \Rr^p(d)_{\Lambda^2} $ be such that 
	\begin{align*}
		\|u_0-u_0^\e \|_{L^n_x} &\le \e \\
		\|b_0-b_0^\e \|_{L^n_x} &\le \e.
	\end{align*}
Let us write $a_1^\e(t)=e^{-t\Sh}u_0^\e$ and $a_2^\e(t)=e^{-t\Mh}b_0^\e$. 
Then 
\begin{equation}
	\|a_1\|_{\CU_T}\le \| a_1-a_1^\e\|_{\CU_T} + \|a_1^\e\|_{\CU_T}.
\end{equation}
Since $\Sh$ generates a bounded semi-group, $\| a_1-a_1^\e\|_{\CU_T}\le K_{\alpha,p} \e$.

and $ \|a_1^\e\|_{\CU_T}=\sup_{0<t<T}\bigl(t^{\frac{\alpha}{2}}\|e^{-t\Sh}u_0^\e \|_{L^p_x}
+t^{\frac{1+\alpha}{2}}\|de^{-t\Sh}u_0^\e\|_{L^p_x}\bigr)\le K_{\alpha,p} T^\frac{\alpha}{2}\|u_0^\e\|_{L^p_x}$. 

We get the same estimates for $b_0$, and combining them together we get
\begin{equation*}
	\| a_1\|_{\CU_T}+\|a_2 \|_{\CB_T}\le \e K_{\alpha,p}\left( T^\frac{\alpha}{2} (\|u_0^\e\|_{L^p_x}+\|b_0^\e\|_{L^p_x})+\e\right).
\end{equation*}
Choosing $T$ small enough, such that $T^\frac{\alpha}{2} (\|u_0^\e\|_{L^p_x}+\|b_0^\e\|_{L^p_x})\le \e$, we get
\begin{equation}
	\| a_1\|_{\CU_T}+\|a_2 \|_{\CB_T}\le 2K_{\alpha,p}\, \e.
\end{equation} 
\end{proof}

The proof of Lemma~\ref{lem_Bilin2} proceeds similarly to the proof of Lemma~\ref{lem_Bilin1}, except for the third estimate $B_3$.
\begin{proof}[Proof of Lemma~\ref{lem_Bilin2}]
	Recall that $\alpha=1-\frac{n}{p}$.
	\begin{enumerate}
		\item Let $\theta=\frac{1-\alpha}{2}$. Then $\frac{2}{p}-\frac{2\theta}{n}=\frac{1}{p}$, so the Sobolev inclusion $W^{2\theta,\frac{p}{2}}\hookrightarrow L^p$ holds.
		 For $u,v\in\CU_T$, by definition of $\CU_T$,
		$s\mapsto s^{\frac{1}{2}+\alpha}u(s)\lrcorner\,dv(s) 
		\in {\mathscr{C}}_b(\demi{0,T};L^{\frac{p}{2}}(\Ri^n,\Lambda^1)$
		with bounded $L^\infty$-norm in time.
		The Leray projector $\Leray$ is bounded from 
		$L^{\frac{p}{2}}(\Ri^n,\Lambda^1)$ to ${\rm{\sf N}}^{\frac{p}{2}}(d^*)_{|_{\Lambda^1}}$, and for $\theta>0$ the operator $\Sh^{-\theta}$ is bounded from $W^{2\theta,\frac{p}{2}}$ to $L^{\frac{p}{2}}$.

		Let $t\in \ouvert{0,T}$. Then we get:
		\begin{align*}
			\left\|B_1(u,v)(t)\right\|_{L^p_x}&=\left\| \int_0^t s^{-\alpha-\frac{1}{2}}\Sh^\theta e^{-(t-s)\Sh} \Sh^{-\theta}\Leray \left(-s^{\frac{\alpha}{2}}u(s)\contr s^{\frac{\alpha+1}{2}}dv(s)\right)\dd s \right\|_{L^p_x} \\
(1)\qquad	&\lesssim \int_0^t s^{-\alpha-\frac{1}{2}}\left\|\Sh^\theta e^{-(t-s)\Sh}\right\|_{L^p\rightarrow L^p} \left\| \Sh^{-\theta}\Leray \left(-s^{\frac{\alpha}{2}}u(s)\contr s^{\frac{\alpha+1}{2}}dv(s)\right)\right\|_{L^p_x} \dd s \\
(2)\qquad	&\lesssim \int_0^t s^{-\alpha-\frac{1}{2}}(t-s)^{-\theta} s^\frac{\alpha}{2} \|u(s)\|_{L^p_x} \ s^\frac{\alpha+1}{2} \|dv(s)\|_{L^\frac{p}{2}_x} \dd s \\
(3)\qquad	&\lesssim \left(\int_0^t s^{-\alpha-\frac{1}{2}}(t-s)^{-\theta} \dd s \right) \|u\|_{\CU_T}\|v\|_{\CU_T} \\
(4)\qquad	&\lesssim t^{-\alpha-\frac{1}{2}-\theta+1}\int^1_0 \sigma^{-\alpha-\frac{1}{2}}(1-\sigma)^{-\theta} \dd \sigma \|u\|_{\CU_T}\|v\|_{\CU_T} \\
(5)\qquad	&\less t^{-\frac{\alpha}{2}}\|u\|_{\CU_T}\|v\|_{\CU_T},
		\end{align*}
		where (1) uses the operator norm of $S^\theta e^{-(t-s)\Sh}$, (2) uses successively the Sobolev injection $W^{2\theta,\frac{p}{2}}\hookrightarrow L^p$, the continuity of $S^{-\theta}$ from $W^{2\theta,\frac{p}{2}}$ to $L^{\frac{p}{2}}$, the continuity of the Leray projector $\Leray$ and the Hölder inequality - the same steps as for Lemma~\ref{lem_Bilin1}. (3) uses simply the definition of the $\CU_T$ norm and (4) and (5) are straightforward integral computations - since $n<p<2n$, both $\alpha+\frac{1}{2}$ and $\theta$ are strictly lower than $1$.
	
		This gives us our first estimate $\sup_{0<t<T}\bigl(t^{\frac{\alpha}{2}}\|B_1(u,v)(t)\|_{L^p_x}\bigl)<+\infty$. 
		
		The second estimate proceeds similarly: taking $\theta=\frac{n}{2p}=\frac{1-\alpha}{2}$ as before, we get
		\begin{align*}
			\left\|dB_1(u,v)(t)\right\|_{L^p_x}&=\left\| \int_0^t s^{-\alpha-\frac{1}{2}}d \Sh^\theta  e^{-(t-s)\Sh} \Sh^{-\theta}\Leray \left(-s^{\frac{\alpha}{2}}u(s)\contr s^{\frac{\alpha+1}{2}}dv(s)\right)\dd s \right\|_{L^p_x} \\
			&\lesssim \left(\int_0^t s^{-\alpha-\frac{1}{2}}(t-s)^{-\theta-\frac{1}{2}} \dd s \right) \|u\|_{\CU_T}\|v\|_{\CU_T} \\
			&\lesssim t^{-\frac{1+\alpha}{2}}\|u\|_{\CU_T}\|v\|_{\CU_T},
		\end{align*}
		with a multiplicative constant independent from $T$. 
		
		This gives us our second estimate $\sup_{0<t<T}\bigl(t^{\frac{1+\alpha}{2}}\|dB_1(u,v)(t)\|_{L^p_x}\bigl)<+\infty$
		\item As for Lemma~\ref{lem_Bilin1}, the proof of point 2. proceeds exactly as in the previous point.
		\item Let $u\in \CU_T$ and $b\in \CB_T$, and set again $\theta=\frac{n}{2p}=\frac{1-\alpha}{2}$.
		Taking the $L^p$ norm of $B_3(u,b)(t)$ now yields
		\begin{align*}
			\|B_3(u,b)(t)\|_{L^p_x}&=\left\| \int_0^t s^{-\alpha} \Mh^\theta e^{-(t-s) \Mh}s^{\alpha}\Mh^{-\theta}\Bigl(-d\bigl(u(s)\contr b(s)\bigr)\Bigr)\,{\rm d}s \right\|_{L^p_x} \\
			&\lesssim \int_0^t s^{-\alpha} \left\| d \Sh^\theta e^{-(t-s)\Sh} \right\|_{L^p\rightarrow L^p}  \left\|\Sh^{-\theta}\bigl(s^\frac{\alpha}{2}u(s)\contr s^\frac{\alpha}{2}b(s)\bigr)\right\|_{L^p_x} \dd s  \\
			&\lesssim \left(\int^t_0 s^{-\alpha}(t-s)^{-\theta-\frac{1}{2}} {\rm d}s\right) \|u\|_{\CU_T}\|b\|_{\CB_T}  \\
			&\lesssim t^{-\frac{\alpha}{2}} \|u\|_{\CU_T}\|b\|_{\CB_T}.
		\end{align*}
For the last term $d^*B_3(u,v)$ we use the Leibniz inequality \eqref{Leibniz} to get 
\begin{equation}
	\left\|d\bigl(u(s)\lrcorner\,b(s)\bigr)\right\|_{\frac{p}{2}}\less \|u(s)\|_{L^p_x}\|d^*b(s)\|_{L^p_x} + \|b(s)\|_{L^p_x}\|du(s)\|_{L^p_x}.
\end{equation}
		We can now use this estimate to compute $\|d^*B_3(u,b)(t)\|_{L^p_x}$, using the same methods as before: 
		
		\begin{align*}
			\|d^*B_3(u,b)(t)\|_{L^p_x}&=\left\| \int_0^t s^{-\alpha-\frac{1}{2}} d^* \Mh^\theta e^{-(t-s) \Mh}s^{\alpha+\frac{1}{2}} \Mh^{-\theta}\Bigl(-d\bigl(u(s)\lrcorner\,b(s)\bigr)\Bigr)\,{\rm d}s \right\|_{L^p_x} \\
			&\lesssim  \int_0^t s^{-\alpha-\frac{1}{2}} \left\|d^* \Mh^\theta e^{-(t-s) \Mh}\right\|_{L^p\rightarrow L^p} s^{\alpha+\frac{1}{2}} \left\| \Mh^{-\theta}\Bigl(-d\bigl(u(s)\lrcorner\,b(s)\bigr)\Bigr)\right\|_{L^p_x}\,{\rm d}s  \\
			&\lesssim  \int_0^t s^{-\alpha-\frac{1}{2}}  (t-s)^{-\theta-\frac{1}{2}}  s^{\alpha+\frac{1}{2}}\left\|\Mh^{-\theta}\Bigl(-d\bigl(u(s)\lrcorner\,b(s)\bigr)\Bigr)\right\|_{W^{2\theta,\frac{p}{2}}}\,{\rm d}s  \\
			&\lesssim  \int_0^t s^{-\alpha-\frac{1}{2}} (t-s)^{-\theta-\frac{1}{2}}  s^{\alpha+\frac{1}{2}} \left\|-d\bigl(u(s)\lrcorner\,b(s)\bigr)\right\|_{\frac{p}{2}}\,{\rm d}s  \\
			&\lesssim  \int_0^t s^{-\alpha-\frac{1}{2}}  (t-s)^{-\theta-\frac{1}{2}}  s^{\alpha+\frac{1}{2}} \left(\|u(s)\|_{L^p_x}\|d^*b(s)\|_{L^p_x} + \|b(s)\|_{L^p_x}\|du(s)\|_{L^p_x}\right) \,{\rm d}s  \\
			&\lesssim \left(\int^t_0 s^{-\alpha-\frac{1}{2}}(t-s)^{-\theta-\frac{1}{2}} {\rm d}s\right) \|u\|_{\CU_T}\|b\|_{\CB_T}  \\
			&\lesssim t^{-\frac{1+\alpha}{2}}\|u\|_{\CU_T}\|b\|_{\CB_T}
		\end{align*}
	\end{enumerate}
	Which concludes our proof of Lemma~\ref{lem_Bilin2}.
\end{proof}

\begin{proof}[Proof of Lemma~\ref{lem:UTBTmild}]
	To prove this lemma, first observe that if $u_0\in {\rm{\sf N}}^n(d^*)_{|_{\Lambda^1}}$ and 
	$b_0\in {\rm{\sf R}}^n(d)_{|_{\Lambda^2}}$, then for all $T>0$,
	$t\mapsto e^{-t\Sh}u_0 \in {\mathscr{C}}_b(\demi{0,T};{\rm{\sf N}}^n(d^*)_{|_{\Lambda^1}})$
	and $t\mapsto e^{-t\Mh}b_0\in  {\mathscr{C}}_b(\demi{0,T};{\rm{\sf R}}^n(d)_{|_{\Lambda^2}})$.
	It remains to show that if $u\in{\mathscr{U}}_T$ and $b\in{\mathscr{B}}_T$, then
	$B_1(u,u)\in {\mathscr{C}}_b(\demi{0,T};{\rm{\sf N}}^n(d^*)_{|_{\Lambda^1}})$, 
	$B_2(b,b)\in {\mathscr{C}}_b(\demi{0,T};{\rm{\sf N}}^n(d^*)_{|_{\Lambda^1}})$
	and $B_3(u,b)\in {\mathscr{C}}_b(\demi{0,T};{\rm{\sf R}}^n(d)_{|_{\Lambda^2}})$.
	
	To prove boundedness, we use the same method as in the previous lemma:
	
	recall that $\alpha=1-\frac{n}{p}$ and chose $\varphi=\frac{n}{p}-\frac{1}{2}$, so that $W^{2\varphi,\frac{p}{2}}\hookrightarrow L^n$. Then we can proceed similarly: 
	\begin{align*}
		\left\|B_1(u,v)(t)\right\|_n&=\left\| \int_0^t s^{-\alpha-\frac{1}{2}}\Sh^\varphi e^{-(t-s)\Sh} \Sh^{-\varphi}\Leray \left(-s^{\frac{\alpha}{2}}u(s)\contr s^{\frac{\alpha+1}{2}}du(s)\right)\dd s \right\|_{L^n_x} \\
		&\lesssim \left(\int_0^t s^{-\alpha-\frac{1}{2}}(t-s)^{-\varphi} \dd s \right) \|u\|_{\CU_T}^2 \\
		&\lesssim t^{-\alpha-\frac{1}{2}-\varphi+1}\int^1_0 \sigma^{-\alpha-\frac{1}{2}}(1-\sigma)^{-\varphi} \dd \sigma \|u\|_{\CU_T}^2\\
		&\lesssim \|u\|_{\CU_T}^2.
	\end{align*}
	The continuity in $0$ is then straightforward, and the terms $B_2$ and $B_3$ can be treated similarly.
\end{proof}

\section{Uniqueness}\label{Uniqueness}

\begin{theorem}\label{thm:uniqueness}
	Let $T\in [0,\infty]$ and assume there exist two solutions $(u_i,b_i)$, $i=1,2$ of \eqref{mhd1diff} with the same initial data $(u_0,b_0)$, and such that 
	\begin{align*}
		d u_i &\in{\mathscr{C}}_b(\demi{0,T};L^\frac{n}{2}(\Ri^n,\Lambda^2)) \\
		d^* b_i &\in{\mathscr{C}}_b(\demi{0,T};L^\frac{n}{2}(\Ri^n,\Lambda^1)).
	\end{align*}
	Then $(u_1,b_1)=(u_2,b_2)$.
\end{theorem}

\begin{remark}
	The condition $(du,d^*b)\in {\mathscr{C}}_b(\demi{0,T};L^{\frac{n}{2}}(\Ri^n,\Lambda^2))\times
	{\mathscr{C}}_b(\demi{0,T};L^{\frac{n}{2}}(\Ri^n,\Lambda^1)))$ in fact implies $(u,b)\in {\mathscr{C}}_b(\demi{0,T};{\rm{\sf N}}^n(d^*)_{|_{\Lambda^1}})\times
	{\mathscr{C}}_b(\demi{0,T};{\rm{\sf R}}^n(d)_{|_{\Lambda^2}})$.
\end{remark}
\begin{proof}
	Assume that there exists $t^*\in \demi{0,\infty}$ such that $(u_1,b_1)=(u_2,b_2)$ on $[0,t^*]$. We write $(u_i,b_i)(t^*,\cdot)=(u_*)$.
	
	Let $u=u_1-u_2$ and $b=b_1-b_2$. For $i=1,2$, since $(u_i,b_i)$ is a solution of $\eqref{mhd1diff}$, we have 
	\begin{align*}
		u_i &= a_1 + B_1(u_i,u_i) + B_2(b_i,b_i) \\
		b_i &= a_2 + B_3(u_i,b_i).
	\end{align*}
	Hence 
	\begin{align}
		u =& B_1(u,u_1)+B_1(u_2,u)+B_2(b,b_1)+B_2(b_2,b) \\
		b =& B_3(u,b_1)+B_3(u_2,b)\,.
	\end{align}

	Let $\e>0$. let $(u_i^\e, b_i^\e)$ be such that $u_i^\e$ and $b_i^\e$ are in $\CC^2[0,T];\CS(\Ri^n))$ with 
	\begin{align}
	&\|d(u_i-u_i^\e)\|_{L^\infty_t(L^\frac{n}{2}_x)} \le \e \label{du_epsilon}\\
	&\|d(b_i-b_i^\e)\|_{L^\infty_t(L^\frac{n}{2}_x)} \le \e \, . \label{db_epsilon}
	\end{align}
Note that this in particular implies 
\begin{align}
	&\|u_i-u_i^\e \|_{L^\infty_t(L^n_x)} \le \e  \label{u_epsilon} \\
	&\|b_i-b_i^\e \|_{L^\infty_t(L^n_x)} \le \e.  \label{b_epsilon} 
\end{align}
We want to prove that there exists some $r>1$ and $\tau>0$ such that 
\begin{align}
	\|du\|_{L^r([t^*,t^*+\tau],L^\frac{n}{2}_x)} &\le K_{r,n,u_i,b_i}\, \e \left( \|du\|_{L^r([t^*,t^*+\tau],L^\frac{n}{2}_x)}+\|d^*b\|_{L^r([t^*,t^*+\tau],L^\frac{n}{2}_x)} \right)\\
	\|d^*b\|_{L^r([t^*,t^*+\tau],L^\frac{n}{2}_x)} &\le K_{r,n,u_i,b_i}\, \e \left( \|d^*b\|_{L^r([t^*,t^*+\tau],L^\frac{n}{2}_x)}+\|du\|_{L^r([t^*,t^*+\tau],L^\frac{n}{2}_x)}\right)
\end{align}
Let $\tau>0$ and $t\in [t^*,t^*+\tau]$.
\begin{enumerate}
	\item Let us look at $dB_1(u,u_1)$ first. We can write $dB_1(u,u_1)= dB_1(u,u_1^\e)+ dB_1(u,u_1-u_1^\e)$. 
	\begin{itemize}
		\item Let us begin with $dB_1(u,u_1-u_1^\e)$. Since $\Leray$ is the projection on $\Nn(d^*)$,
	\begin{equation*}
dB_1(u,u_1-u_1^\e)(t)= \int_{t^*}^t \Dh e^{-(t-t^*-s)\Sh} \Leray\big(u(s)\contr d(u_1-u_1^\e)(s)\big) \dd s = \Dh B_1(u,u_1-u_1^\e)(t).
	\end{equation*}
Using the fact that $\|\Dh f\|_r \sim \|\Sh^\frac{1}{2}f\|_r$ for all $r\in ]1,\infty[$, we can estimate $\Sh^\frac{1}{2}B_1(u,u_1-u_1^\e)$ instead of $dB_1(u,u_1-u_1^\e)$. Using the maximal regularity Theorem~\ref{Sylvie_disciple} we get:
	\begin{equation*}
		\int_{t^*}^t \Sh^\frac{1}{2} e^{-(t-s)\Sh} \Leray\big(u(s)\contr d(u_1-u_1^\e)(s)\big) \dd s=\int_{t^*}^t \Sh e^{-(t-s)\Sh} \Sh^{-\frac{1}{2}}\Leray\big(u(s)\contr d(u_1-u_1^\e)(s)\big) \dd s,
	\end{equation*}
	So
	\begin{align*}
		\left\|\int_{t^*}^t \Sh^\frac{1}{2} e^{-(t-s)\Sh} \Leray\big(u(s)\contr d(u_1-u_1^\e)(s)\big) \dd s \right\|_{L^r(\ouvert{t^*,t^*+\tau}) L^\frac{n}{2}_x} &\less_{r,n} \left\|\Sh^{-\frac{1}{2}} \Leray\big(u\contr d(u_1-u_1^\e)\big)\right\|_{L^r_tL^\frac{n}{2}_x} \\
		&\less_{r,n} \left\| \Leray\big(u\contr d(u_1-u_1^\e)\big)\right\|_{L^r_t W^{1,\frac{n}{3}}_x} \\
		&\less_{r,n} \|u\|_{L^r_tL^n_x} \|d(u_1-u_1^\e)\|_{L^\infty_t L_x^\frac{n}{2}} \\
		&\less_{r,n} \e \|du\|_{L^r_t L^\frac{n}{2}_x}.
	\end{align*}
Hence there exists some constant $K_{r,n}$ such that
\begin{equation}\label{B11}
\|dB_1(u,u_1-u_1^\e)\|_{L^r_tL^\frac{n}{2}_x}\le K_{r,n} \e \|du\|_{L^r_t L^\frac{n}{2}_x}.
\end{equation}
	\item The second term does not require maximal regularity:
	\begin{align*}
		\left\| dB_1(u,u_1^\e)(t) \right\|_{\frac{n}{2}}&\less_n \int_{t^*}^t \left\| d e^{-(t-s)\Sh } \right\|_{L^{\frac{n}{2}}\rightarrow L^{\frac{n}{2}}} \left\| \Leray \big( u(s)\contr du_1^\e(s)\big) \dd s \right\|_{\frac{n}{2}} \\
		&\less \int_{t^*}^t \frac{1}{\sqrt{t-s}} \|u(s)\|_n \| du_1^\e(s) \|_n \dd s\, .
	\end{align*}
	Using the convolution injection $L^1\star L^r\hookrightarrow L^r$ we get:
	\begin{equation}\label{B12}
		\left\| dB_1(u,u_1^\e)\right\|_{L^r([0,\tau[;L^\frac{n}{2})}\less_{r,n} 2\sqrt{\tau}\, \|d u\|_{L^r([t^*,t^*+\tau[;L^\frac{n}{2})} \|d u_1^\e \|_{L^\infty(L^n_x)}.
	\end{equation}
	\end{itemize}
Now $\|d u_1^\e \|_{L^\infty(L^n)}$ is well defined but not necessarily bounded as $\e$ goes to $0$. However we can always pick $\tau$ small enough to ensure that $\sqrt{\tau} \|d u_1^\e \|_{L^\infty(L^n)}\le \e$.
Therefore combining estimates \eqref{B11} and \eqref{B12} there exists a constant $K_{r,n,u_i}$ such that
\begin{equation}\label{B1_uniqueness}
 \left\| dB_1(u,u_1) \right\|_{L^r([t^*,t^*\tau[;L^\frac{n}{2})} \le K_{r,n,u_i} \e  \|d u\|_{L^r([t^*,t^*+\tau[;L^\frac{n}{2})}.
\end{equation}

\item Let us write $dB_1(u_2,u)=dB_1(u_2-u_2^\e,u)+dB_1(u_2^\e,u)$.
Then by maximal regularity we get:
\begin{align*}
	\|dB_1(u_2-u_2^\e,u)\|_{L^r_t L^\frac{n}{2}}&\less_{r,n} \left\|\Sh^{-\frac{1}{2}} \Leray\big((u_2-u_2^e)\contr du\big)\right\|_{L^r_tL^\frac{n}{2}_x} \\
	&\less_{r,n} \|du\|_{L^r_tL^\frac{n}{2}_x} \|u_2-u_2^\e\|_{L^\infty_t L^n} \\
	&\less_{r,n} \e \|du\|_{L^r_t L^\frac{n}{2}_x}.
\end{align*}
Besides 
\begin{align*}
		\left\| dB_1(u_2^\e,u)(t) \right\|_{\frac{n}{2}}&\less_{n} \int_{t^*}^t \left\| d e^{-(t-s)\Sh } \right\|_{L^{\frac{n}{2}}\rightarrow L^{\frac{n}{2}}} \left\| \Leray \big( (u_2^\e)(s)\contr du(s)\big) \dd s \right\|_{\frac{n}{2}} \\
	&\less_n \int_{t^*}^t \frac{1}{\sqrt{t-s}} \|u_2^\e(s)\|_\infty \| du(s) \|_\frac{n}{2} \dd s \, ,\\
\end{align*}
And by convolution we get
\begin{equation}
	\left\| dB_1(u_2^\e,u)(t) \right\|_{\frac{n}{2}}\less_{r,n} \sqrt{\tau} \|u_2^\e\|_{L^\infty_tL^\infty_x} \| du \|_{L^r_tL^\frac{n}{2}_x}.
\end{equation}
Setting $\tau$ such that $\sqrt{\tau} \|u_2^\e\|_{L^\infty_tL^\infty_x}\le \e$, we finally get
\begin{equation}
	\left\| dB_1(u_2,u) \right\|_{L^r([t^*,t^*+\tau[;L^\frac{n}{2})} \le K_{r,n,u_i} \e  \|d u\|_{L^r([t^*,t^*+\tau[;L^\frac{n}{2})}.
\end{equation}
\item The next terms $B_2(b,b_1)$ and $B_2(b_2,b)$ are treated in the exact same way. 
\item Recall that $b=B_3(u,b_1)+B_3(u_2,b)$.
We start by writing $d^*B_3(u,b_1)=d^*B_3(u,b_1-b_1^\e)+d^*B_3(u,b_1^\e)$. 
\begin{itemize}
	\item Let us recall from section \ref{LqLp} that $d^*e^{-(t-s)\Mh} d=\Sh e^{-(t-s)\Sh}$. Using the maximal regularity property \ref{Sylvie_disciple}   we then get:
	\begin{equation}\label{UD1}
		\| d^*B_3(u,b_1-b_1^\e) \|_{L^r_tL^\frac{n}{2}_x} \le \| u\contr (b_1-b_1^\e)\|_{L^r_tL^\frac{n}{2}_x} \less_{r,n} \|u\|_{L^r_tL^n_x} \|(b_1-b_1^\e)\|_{L^\infty_tL^n_x} \less_{r,n} \e \|du\|_{L^r_tL^\frac{n}{2}_x}.
	\end{equation} 
	\item Using the Leibniz-rule \eqref{Leibniz}, we get: 
	\begin{align*}
		\|d^*B_3(u,b_1^\e)(t)\|_\frac{n}{2} &\less_n \int_{t^*}^t \frac{1}{\sqrt{t-s}} \|d(u\contr b_1^e)\|_{\frac{n}{2}}(s)ds\\
		&\less_n \int_{t^*}^t \frac{1}{\sqrt{t-s}} \left(\|u(s)\|_n\|db_1^\e(s)\|_n+\|du(s)\|_{\frac{n}{2}}\|b_1^\e(s)\|_\infty\right),
	\end{align*}
And by convolution we can conclude that 
\begin{equation}\label{U=2}
	\|d^*B_3(u,b_1^e)\|_{L^r_tL^\frac{n}{2}_x} \less_{r,n} \sqrt{\tau} \left(\|b_1^\e\|_{L^\infty_tL^\infty_x}+\|db_1^\e\|_{L^\infty_tL^n_x}\right) \|du\|_{L^r_tL^\frac{n}{2}_x}.
\end{equation}
\end{itemize}
Choosing $\tau$ such that $\sqrt{\tau} \left(\|b_1^\e\|_{L^\infty_tL^\infty_x}+\|db_1^\e\|_{L^\infty_tL^n_x}\right)\le \e$ let us finally get
\begin{equation}
	\|d^*B_3(u,b_1)\|_{L^r_tL^\frac{n}{2}_x} \le K_{r,n,u_i,b_i}\, \e \|du\|_{L^r_tL^\frac{n}{2}_x}.
\end{equation}
A similar computation shows that 
\begin{equation}
	\|d^*B_3(u_2,b)\|_{L^r_tL^\frac{n}{2}_x} \le K_{r,n,u_i,b_i}\, \e \|d^*b\|_{L^r_tL^\frac{n}{2}_x},
\end{equation}
and with this last estimate we have proven that
\begin{align}
	\left\| du \right\|_{L^r([t^*,t^*+\tau[;L^\frac{n}{2})} &\le K_{r,n,u_i,b_i}\, \e \left( \|d u\|_{L^r([t^*,t^*+\tau[;L^\frac{n}{2})}+ \|d^* b\|_{L^r([t^*,t^*+\tau[;L^\frac{n}{2})} \right)\label{contraction}\\
	\left\| d^* b \right\|_{L^r([t^*,t^*+\tau[;L^\frac{n}{2})} &\le K_{r,n,u_i,b_i} \e \left(\|d u\|_{L^r([t^*,t^*+\tau[;L^\frac{n}{2})} + \|d^* b\|_{L^r([t^*,t^*+\tau[;L^\frac{n}{2})} \right)\nonumber,
\end{align}
where $K_{r,n,u_i,b_i}$ is a constant independent of $u$, and $\e$ can be chosen arbitrarily small.
Then letting $\e$ be such that $K_{r,n,u_i,b_i}\e \le \frac{1}{4}$ we get
\begin{equation*}
	\left\| du \right\|_{L^r([t^*,t^*+\tau[;L^\frac{n}{2})} + \left\| d^* b \right\|_{L^r([t^*,t^*+\tau[;L^\frac{n}{2})} \le \frac{1}{2} \left(	\left\| du \right\|_{L^r([t^*,t^*+\tau[;L^\frac{n}{2})} + \left\| d^* b \right\|_{L^r([t^*,t^*+\tau[;L^\frac{n}{2})}\right),
\end{equation*}
which proves that $du$ and $d^*b$ (and hence $u$ and $b$) are equal to $0$ on $[t^*,t^*+\tau[$.
Let 
\begin{equation*}
	I=\left\{t^*, \text{the system \eqref{mhd1diff} has a unique solution on } [0,t^*[.\right\}
\end{equation*}
Then $I$ is open, and it is also closed by continuity. Thus $I=[0,T[$ by connectedness.  
\end{enumerate}
\end{proof}


\begin{bibdiv}
\begin{biblist}

\end{biblist}
\end{bibdiv}


\begin{thebibliography}{11}
	
	
	\bib{AB20}{article}{
		AUTHOR = {Amrouche, Ch\'{e}rif},
		author={Boukassa, Saliha},
		TITLE = {Existence and regularity of solution for a model in
			magnetohydrodynamics},
		JOURNAL = {Nonlinear Anal.},
		VOLUME = {190},
		YEAR = {2020},
		PAGES = {111602, 20},
	}
	
	\bib{AKMcI06Invent}{article}{
		AUTHOR = {Axelsson, Andreas},
		author={Keith, Stephen},
		author={McIntosh, Alan},
		TITLE = {Quadratic estimates and functional calculi of perturbed
			{D}irac operators},
		JOURNAL = {Invent. Math.},
		VOLUME = {163},
		YEAR = {2006},
		NUMBER = {3},
		PAGES = {455--497},
	}
	
	\bib{AMcI04}{incollection}{
		AUTHOR = {Axelsson, Andreas},
		author={McIntosh, Alan},
		TITLE = {Hodge decompositions on weakly {L}ipschitz domains},
		BOOKTITLE = {Advances in analysis and geometry},
		SERIES = {Trends Math.},
		PAGES = {3--29},
		PUBLISHER = {Birkh\"auser, Basel},
		YEAR = {2004},
	}
	
	\bib{BCD11}{book}{
		AUTHOR = {Bahouri, Hajer and Chemin, Jean-Yves and Danchin, Rapha\"{e}l},
		TITLE = {Fourier analysis and nonlinear partial differential equations},
		SERIES = {Grundlehren der mathematischen Wissenschaften [Fundamental
			Principles of Mathematical Sciences]},
		VOLUME = {343},
		PUBLISHER = {Springer, Heidelberg},
		YEAR = {2011},
		PAGES = {xvi+523},
		ISBN = {978-3-642-16829-1},
		MRCLASS = {35-02 (35L72 35Q30 42-02 42B37 76B03 76D03 76N10)},
		MRNUMBER = {2768550},
		MRREVIEWER = {Peter R. Massopust},
		DOI = {10.1007/978-3-642-16830-7},
		URL = {https://doi.org/10.1007/978-3-642-16830-7},
	}
	
	\bib{BH20}{article}{
		AUTHOR = {Brandolese, Lorenzo},
		author={He, Jiao},
		TITLE = {Uniqueness theorems for the {B}oussinesq system},
		JOURNAL = {Tohoku Math. J. (2)},
		VOLUME = {72},
		YEAR = {2020},
		NUMBER = {2},
		PAGES = {283--297},
	}
	
	\bib{BM20}{unpublished}{
		title={Well-posedness for the Boussinesq system in critical spaces via maximal regularity}, 
		author={Brandolese, Lorenzo},
		author={Monniaux, Sylvie},
		year={2020},
		eprint={2012.02450},
		note={Preprint arXiv 2012.02450},
	}
	
	\bib{CMcI10}{article}{
		AUTHOR = {Costabel, Martin},
		author={McIntosh, Alan},
		TITLE = {On {B}ogovski\u{\i} and regularized {P}oincar\'{e} integral operators
			for de {R}ham complexes on {L}ipschitz domains},
		JOURNAL = {Math. Z.},
		VOLUME = {265},
		YEAR = {2010},
		NUMBER = {2},
		PAGES = {297--320},
	}
	
	\bib{FK64}{article}{
		AUTHOR = {Fujita, H.}, 
		author={Kato, T.},
		TITLE = {On the {N}avier-{S}tokes initial value problem. {I}},
		JOURNAL = {Arch. Rational Mech. Anal.},
		VOLUME = {16},
		YEAR = {1964},
		PAGES = {269--315},
	}
	
	\bib{Gr08}{book}{
		AUTHOR = {Grafakos, L.},
		TITLE = {Classical {F}ourier analysis},
		SERIES = {Graduate Texts in Mathematics},
		VOLUME = {249},
		EDITION = {Second},
		PUBLISHER = {Springer, New York},
		YEAR = {2008},
		PAGES = {xvi+489},
		ISBN = {978-0-387-09431-1},
		MRCLASS = {42-01 (42Bxx)},
		MRNUMBER = {2445437},
		MRREVIEWER = {Andreas Seeger},
	}
	
	\bib{HMM}{article}{
		AUTHOR = {Hofmann, Steve},
		author={Mitrea, Marius},
		author={Monniaux, Sylvie},
		TITLE = {Riesz transforms associated with the {H}odge {L}aplacian in
			{L}ipschitz subdomains of {R}iemannian manifolds},
		JOURNAL = {Ann. Inst. Fourier (Grenoble)},
		VOLUME = {61},
		YEAR = {2011},
		NUMBER = {4},
		PAGES = {1323--1349 (2012)},
	}
	
	\bib{LSU68}{book}{
		AUTHOR = {Lady\v{z}enskaja, O. A.}
		author = {Solonnikov, V. A.}
		author = {Ural'ceva, N.N.},
		TITLE = {Linear and quasilinear equations of parabolic type},
		SERIES = {Translations of Mathematical Monographs, Vol. 23},
		NOTE = {Translated from the Russian by S. Smith},
		PUBLISHER = {American Mathematical Society, Providence, R.I.},
		YEAR = {1968},
		PAGES = {xi+648},
		MRCLASS = {35.62},
		MRNUMBER = {0241822},
		MRREVIEWER = {B. Frank Jones, Jr.},
	}
	
	\bib{Lem02}{book}{
		author={Lemari{\'e}-Rieusset, Pierre Gilles},
		title={Recent developments in the Navier-Stokes problem},
		series={Chapman \& Hall/CRC Research Notes in Mathematics},
		volume={431},
		publisher={Chapman \& Hall/CRC, Boca Raton, FL},
		date={2002},
		pages={xiv+395},
	}
	
	\bib{McIM18}{article}{
		AUTHOR = {McIntosh, Alan},
		author={Monniaux, Sylvie},
		TITLE = {Hodge-{D}irac, {H}odge-{L}aplacian and {H}odge-{S}tokes
			operators in {$L^p$} spaces on {L}ipschitz domains},
		JOURNAL = {Rev. Mat. Iberoam.},
		VOLUME = {34},
		YEAR = {2018},
		NUMBER = {4},
		PAGES = {1711--1753},
	}
	
	\bib{MMM08}{article}{
		AUTHOR = {Mitrea, Dorina},
		author={Mitrea, Marius},
		author={Monniaux, Sylvie},
		TITLE = {The {P}oisson problem for the exterior derivative operator
			with {D}irichlet boundary condition in nonsmooth domains},
		JOURNAL = {Commun. Pure Appl. Anal.},
		VOLUME = {7},
		YEAR = {2008},
		NUMBER = {6},
		PAGES = {1295--1333},
	}
	
	\bib{MM09a}{article}{
		AUTHOR = {Mitrea, Marius},
		author={Monniaux, Sylvie},
		TITLE = {On the analyticity of the semigroup generated by the {S}tokes
			operator with {N}eumann-type boundary conditions on
			{L}ipschitz subdomains of {R}iemannian manifolds},
		JOURNAL = {Trans. Amer. Math. Soc.},
		VOLUME = {361},
		YEAR = {2009},
		NUMBER = {6},
		PAGES = {3125--3157},
	}
	
	\bib{MM09b}{article}{
		AUTHOR = {Mitrea, Marius},
		author={Monniaux, Sylvie},
		TITLE = {The nonlinear {H}odge-{N}avier-{S}tokes equations in
			{L}ipschitz domains},
		JOURNAL = {Differential Integral Equations},
		VOLUME = {22},
		YEAR = {2009},
		NUMBER = {3-4},
		PAGES = {339--356},
	}
	
	\bib{M06}{article}{
		AUTHOR = {Monniaux, Sylvie},
		TITLE = {Navier-{S}tokes equations in arbitrary domains: the
			{F}ujita-{K}ato scheme},
		JOURNAL = {Math. Res. Lett.},
		VOLUME = {13},
		YEAR = {2006},
		NUMBER = {2-3},
		PAGES = {455--461},
	}
	
	\bib{Mo21}{article}{
		TITLE = {{Existence in critical spaces for the magnetohydrodynamical system in 3D bounded Lipschitz domains}},
		AUTHOR = {Monniaux, Sylvie},
		URL = {https://hal.archives-ouvertes.fr/hal-03158650},
		NOTE = {working paper or preprint},
		YEAR = {2021},
		MONTH = {Mars},
		PDF = {https://hal.archives-ouvertes.fr/hal-03158650/file/mhdHAL.pdf},
	}
	
	\bib{Schw95}{book}{
		AUTHOR = {Schwarz, G{\"u}nter},
		TITLE = {Hodge decomposition---a method for solving boundary value
			problems},
		SERIES = {Lecture Notes in Mathematics},
		VOLUME = {1607},
		PUBLISHER = {Springer-Verlag, Berlin},
		YEAR = {1995},
		PAGES = {viii+155},
	}
	
	
	\bib{ST83}{article}{
		AUTHOR = {Sermange, Michel},
		author={Temam, Roger},
		TITLE = {Some mathematical questions related to the {MHD} equations},
		JOURNAL = {Comm. Pure Appl. Math.},
		VOLUME = {36},
		YEAR = {1983},
		NUMBER = {5},
		PAGES = {635--664},
	}


\end{thebibliography}
\end{document}